\documentclass[12pt,noinfoline]{imsart}

\RequirePackage[OT1]{fontenc}
\RequirePackage{amsthm,amsmath}
\RequirePackage[numbers]{natbib}

\RequirePackage[colorlinks,citecolor=blue,urlcolor=blue]{hyperref}

\usepackage{verbatim}
\usepackage{bbold}
\usepackage{float}
\usepackage{ulem}

\usepackage{algorithm}
\usepackage{algpseudocode}

\usepackage{subcaption}

\allowdisplaybreaks

\usepackage{fullpage}
\usepackage{bbm}

\usepackage{layout}
\usepackage{enumerate}

\usepackage{dsfont}
\usepackage[mathscr]{eucal}
\usepackage[toc,page]{appendix}
\usepackage{mathrsfs}
\usepackage{color}
\usepackage{pifont}
\usepackage{bm}
\usepackage{latexsym}
\usepackage{amsfonts}
\usepackage{amssymb}
\usepackage{epsfig}
\usepackage{graphicx}
\usepackage{multirow}

\newtheorem{theorem}{Theorem}[section]

\newtheorem{lemma}[theorem]{Lemma}
\newtheorem{definition}[theorem]{Definition}

\newtheorem{remark}[theorem]{Remark}

\newtheorem{assumption}[theorem]{Assumption}

\usepackage{verbatim}

\usepackage{tikz}
\usetikzlibrary{fit,positioning,arrows,automata,calc}
\tikzset{
	main/.style={circle, minimum size = 5mm, thick, draw =black!80, node distance = 10mm},
	connect/.style={-latex, thick},
	box/.style={rectangle, draw=black!100}
}

\usepackage{mathtools}
\usepackage{tabu}

\hypersetup{
	colorlinks=true,       
	linkcolor=blue,        
	citecolor=blue,        
	filecolor=magenta,     
	urlcolor=blue         
}
\newcommand*\diff{\mathop{}\!\mathrm{d}}

\DeclarePairedDelimiter\floor{\lfloor}{\rfloor}

\newcommand{\bk}{\boldsymbol{k}}
\newcommand{\bb}{\boldsymbol{b}}

\newcommand{\E}{\mathbb{E}}
\newcommand{\R}{\mathbb{R}}
\newcommand{\Var}{\mathrm{Var}}

\newcommand\norm[1]{\left\lVert#1\right\rVert}
\newcommand{\id}{\mathrm{id}}

\DeclareMathOperator*{\argmin}{argmin}

\begin{document}

\begin{frontmatter}

\title{{\large Transportation of Measure Regression in Higher Dimensions}}

\runtitle{Transportation of Measure Regression in Higher Dimensions}

\begin{aug}
\author{\fnms{Laya} \snm{Ghodrati} \quad \ead[label=e1]{laya.ghodrati@epfl.ch}} \and \quad
\author{\fnms{Victor M.} \snm{Panaretos}\ead[label=e2]{victor.panaretos@epfl.ch} \\ \texttt{laya.ghodrati@epfl.ch} \quad\,\, \texttt{victor.panaretos@epfl.ch}}


\runauthor{L. Ghodrati \& V.M. Panaretos}

\affiliation{Ecole Polytechnique F\'ed\'erale de Lausanne}

\address{Institut de Math\'ematiques\\
Ecole Polytechnique F\'ed\'erale de Lausanne}

\end{aug}

\begin{abstract} 
	We present an optimal transport framework for performing regression when both the covariate and the response are probability distributions on a compact Euclidean subset $\Omega\subset\mathbb{R}^d$, where $d>1$. Extending beyond compactly supported distributions, this method also applies when both the predictor and responses are Gaussian distributions on $\mathbb{R}^d$. Our approach generalizes an existing transportation-based regression model to higher dimensions. This model postulates that the conditional Fr\'echet mean of the response distribution is linked to the covariate distribution via an optimal transport map. We establish an upper bound for the rate of convergence of a plug-in estimator.
	We propose an iterative algorithm for computing the estimator, which is based on DC (Difference of Convex Functions) Programming. In the Gaussian case, the estimator achieves a parametric rate of convergence, and the computation of the estimator simplifies to a finite-dimensional optimization over positive definite matrices, allowing for an efficient solution. The performance of the estimator is demonstrated in a simulation study.
\end{abstract}

\begin{keyword}[class=AMS]
\kwd[Primary ]{62R20, 62R10}
\kwd[; secondary ]{60G57}
\end{keyword}

\begin{keyword}
\kwd{distributional regression}
\kwd{Fr\'echet mean}
\kwd{Wasserstein space}
\end{keyword}

\end{frontmatter}

\section{Introduction}\label{intro}

Data sets that can be modeled as collections of probability distributions are increasingly appearing in modern applications \citep{schiebinger2019optimal,cazelles2017log,roding2016power,rabin2014adaptive}. This calls for the development of statistical theory and methods in sample/parameter spaces comprising spaces of \textit{measures}. These spaces present with significant challenges for statistical inference, as they are neither finite-dimensional nor linear. The Wasserstein space, in particular, has arisen as a canonical choice of ambient space for such analyses. It exhibits interesting properties that to some extent are analogous to what one has in Hilbert spaces. In that sense it lends itself for statistical methodology, particularly for regression analysis with distributional covariates and responses. For example, Chen et al. \cite{chen2021wasserstein} and Zhang et al. \cite{zhang2022wasserstein} use the tangent structure of the Wasserstein space to develop a tangential Hilbert-type linear model.  \citet{ghodrati2022distribution}, on the other hand, use a shape-constraint approach to define regression via monotone transportation of mass. However, thus far, distributional regression in the Wasserstein space has been confined to measures on the real line. This is  due to the geometrical, computational, and statistical complexities associated with higher dimensions: The absence of closed-form solutions for optimal transport maps, the positive curvature of the space, and the curse of dimensionality are among the obstacles encountered. Despite these challenges, and at least in principle, both models have the potential to be studied in higher dimensions\footnote{Indeed, shortly after our own preprint appeared, \citet{okano2023distribution} also posted a preprint studying a distributional regression model based  on the tangent structure of Wasserstein spaces, but restricted to the special case where the predictor and response are Gaussian distributions on $\R^d$. Compare to our Section \ref{Gaussian Case}.}. 

In this paper, we consider the distributional regression problem in higher dimensions, focussing on the monotone map approach, in light of its leaner technical assumptions and greater statistical interpretability. We incorporate and adapt concepts and techniques from prior studies in higher dimensional statistical optimal transport. Specifically, we employ strategies that address the curse of dimensionality in estimating optimal transport maps by imposing regularity conditions borrowed from \citet{gunsilius2022convergence} and \citet{hutter2019minimax}. And, we draw from previous work on imposing geometric and shape constraints to derive the rate of convergence of Fr\'echet mean (namely Ahidar-Coutrix et al. \cite{ahidar2020convergence} and \citet{chewi2020gradient}). We thus establish identifiability of the monotone map model in higher dimensions, introduce a regularised Fr\'echet-least-squares estimator, and establish its consistency and rate of convergence.{ Given that the estimator is formulated as a minimizer of an objective function, expressible as the difference between two convex functionals over a specific space of optimal maps, we propose an iterative algorithm for its computation using DC (Difference of Convex Functions) Programming.}

{In the case where both predictors and responses are Gaussian distributions, the estimator achieves a parametric convergence rate. Here, the existence of a closed-form solution for optimal transport maps between Gaussian distributions simplifies the objective function for estimator minimization. Consequently, the general iterative DC algorithm for computing the estimator is reduced to DC Programming over positive definite matrices, enhancing computational efficiency. The estimator's performance is further illustrated through a simulation study.}

\section{Background on Optimal Transport and Some Notation}{\label{Wasserstein}}
In this section, we provide some background on optimal transport, and also introduce some notation. For a more comprehensive background see, e.g. \cite{panaretos2020invitation}.  Let $\Omega\subseteq\R^d$ and $\mathcal{W}_2(\Omega)$ be the set of Borel probability measures on $\Omega$ with finite second moment. Let $\mathcal{W}_{2,\mathrm{ac}}(\Omega)$ denote the subset of measures in $\mathcal{W}_2(\Omega)$ that are absolutely continuous with respect to the Lebesgue measure. For two measures, $\mu,\nu \in \mathcal{W}_2(\Omega)$, let $\Gamma(\nu,\mu)$ be the set of couplings of $\mu$ and $\nu$, i.e. the set of Borel probability measures on $\Omega \times \Omega$ with marginals $\nu$ and $\mu$. 
 The 2-Wasserstein distance $d^2_{\mathcal{W}}$ between $\mu,\nu \in \mathcal{W}_2(\Omega)$ is defined as 
\begin{equation}{\label{wdist}}
	d^2_{\mathcal{W}}(\nu,\mu):=\underset{\gamma \in \Gamma(\nu,\mu)}{\inf} \int_{\Omega} \norm{x-y}^2 \diff \gamma(x,y).
\end{equation}
It can be shown that $\mathcal{W}_2(\Omega)$ endowed with $d_{\mathcal{W}}$ is a metric space, which we simply call the Wasserstein space of distributions.

A coupling $\gamma$ is deterministic if it is the joint distribution of $\{X, T(X)\}$ for some deterministic map $T: \Omega \to \Omega$. In such a case, we write $\nu=T\#\mu$ and say that $T$ pushes $\mu$ forward to $\nu$, i.e. $\nu(B)=\mu\{T^{-1}(B)\}$ for any Borel set $B$. Brenier's theorem states that when the source distribution $\mu$ is absolutely continuous with respect to the Lebesgue measure, then there is a unique coupling that achieves the minimum in \eqref{wdist} and is a deterministic coupling induced by a map $T$ which is $\mu$-almost surely the gradient of a convex function $\varphi:\R^d\to \R$, i.e. $T=\nabla \varphi$. The map $T$ and the convex function $\varphi$ are correspondingly termed as the optimal transport map and the Kantorovich potential between the distributions $\mu$ and $\nu$.

 A notion of average of probability distributions can be defined via the Fr\'echet mean with respect to the Wasserstein metric. Let $\Lambda$ be a random measure in $\mathcal{W}_2(\mathbb{R}^d)$, and denote the law of $\Lambda$ as $P$. A Fr\'echet mean (or barycenter) of $\Lambda$ is a minimizer of the Fr\'echet functional
\begin{equation}{\label{frechet-functional}}
F(b)=\frac{1}{2}E d^2_{\mathcal{W}}(b,\Lambda)=\frac{1}{2}\int_{\mathcal{W}_2(\Omega)}  d^2_{\mathcal{W}}(b,\lambda)\diff P(\lambda)\quad b \in \mathcal{W}_2(\mathbb{R}^d).
\end{equation}
If $\Lambda$ has a finite Fr\'echet functional and is absolutely continuous with positive probability, the Fr\'echet mean exists and is unique. 

In this paper, we occasionally use the fact that for any $\mu,\nu, b \in \mathcal{W}_2(\R^d)$ and optimal maps $T_1,T_2$ such that $T_1\#b = \mu$ and $T_2\#b=\nu$ we have:
\begin{equation}\label{PC-d}
d_{\mathcal{W}}(\mu,\nu)\leq\norm{T_1 - T_2}_{L^2(b)}.  
\end{equation}
To see this, recall the definition of the Wasserstein space and note that $(T_1,T_2)\#b$ is a coupling of $\mu$ and $\nu$. The inequality becomes equality whenever $\mu$ and $\nu$ are equal.

\paragraph{Notation.} We use the notation $a \lesssim b$ to indicate that there exists a positive constant $C$ for which $a\leq C b$ holds. We use $\text{diam}(A)$ to denote the diameter of a subset $A \subset \R^d$. We denote the topological interior of a set $A$ by $A^0$. The Euclidean norm of a vector in $\R^d$ is denoted by $\norm{\cdot}$. Given a measure $\mu$,  the $L^p$ norm of a function $f$ with respect to $\mu$ is written as $\norm{f}_{L^p(\mu)}$.
We use the symbol $\nabla f$ for the gradient of a function $f$, and $Df$ for the derivative. 
A multi-index $\bk$ is a vector with integer coordinates $(k_1,\cdots,k_d)$. We write $|\bk|=\sum_{i=1}^d k_i$. For a given multi-index $\bk= (k_1,\cdots,k_d)$, we define the differential operator
$$D^{\bk}=\frac{\partial^{|\bk|}}{\partial x_1^{k_1}\cdots \partial x_d^{k_d}}.$$
The class of H\"older functions of smoothness $\beta>0$ on a set $\Omega$, whose precise definition is recalled later, is denoted by $C^{\beta}(\Omega)$.
The class of continuous maps on $\Omega$ is denoted by $C(\Omega)$. The \textit{identity map} on $\mathbb{R}^d$ will be denoted as $\id(x) = I x = x$, where $I$ is the identity \textit{matrix} on $\mathbb{R}^d$.
Occasionally, we denote by $T_{\mu\to\nu}$ the optimal map between $\mu$ and $\nu$.

\section{The Model and its Identifiability}{\label{The Model and The Identifiability}}

{Let $\Omega \subseteq\R^d$ be convex with a non-empty interior. For the results in this section, $\Omega$ need not be compact.} Let $(\mu,\nu)$ be a pair of random elements in $\mathcal{W}_2(\Omega) \times \mathcal{W}_2(\Omega)$ with a joint distribution denoted by $P$. The \textit{regression operator}, $\Gamma:\mathcal{W}_2(\Omega) \to \mathcal{W}_2(\Omega)$, is the minimizer of the conditional Fr\'echet functional, viewed as a function of $\mu$:
$$
\argmin_b \int_{\mathcal{W}_2(\Omega)} d^2_{\mathcal{W}}(b,\nu)\diff P(\nu\,|\,\mu)=\Gamma(\mu). 
$$
It is implicitly assumed that the Fr\'echet mean of the conditional probability distribution $P(\cdot \,|\, \mu)$ of $\nu$ given $\mu$ is unique for any $\mu$. This uniqueness can be ensured through suitable regularity assumptions on the pair $(\mu,\nu)$. A regression model consists in positing a specific structure for $\Gamma(\mu)$. The identifiability of such a model will typically require additional assumptions on the pair $(\mu,\nu)$. 

We consider generalizing the regression model of \citet{ghodrati2022distribution}, which is defined for distributions on $\R$: We assume that $\Gamma(\mu)=T_0\#\mu$, where $T_0$ is an optimal map, and the response distribution $\nu$ further deviates from its conditional Fr\'echet mean $T_0\#\mu$ by means of a random optimal map perturbation (with Bochner mean identity). 

More explicitly, we consider the regression model
\begin{equation}{\label{model-d}}
   \nu_{i}=T_{\epsilon_i}\#(T_0\#\mu_i),  \quad  \{\mu_i,\nu_i\}_{i=1}^N,
\end{equation}
{where $\{\mu_i,\nu_i\}_{i=1}^N$ are an independent collection of regressor/response pairs in $\mathcal{W}_2(\Omega)\times \mathcal{W}_2(\Omega)$}, $T_0:\Omega \to \Omega $ is an unknown transport map and the $\{T_{\epsilon_i}\}_{i=1}^{N}$ are independent and identically distributed random optimal transport maps with $\E(T_{\epsilon_i})=\id$, representing the noise in our model. The regression task is to estimate the unknown map $T_0$ from the observations $\{\mu_i,\nu_i\}_{i=1}^N$.

Let $P$ denote the joint distribution on $\mathcal{W}_2(\Omega) \times \mathcal{W}_2(\Omega)$ induced by model \eqref{model-d}. We denote by $P_{M}$ the induced marginal distribution of $\mu$ and will be assuming that it is supported on regular measures:  

\begin{assumption}{\label{assumptionMu}}
Let $\mu$ be a measure in the support of $P_M$. Then $\mu$ is absolutely continuous with respect to the Lebesgue measure.
\end{assumption}
\noindent The linear average (Bochner mean) of $P_M$, will be denoted as
$$Q(A)=\int_{\mathcal{W}_2(\Omega)} \mu(A)\diff P_M(\mu).$$
Recall that a twice differentiable function $\varphi: \mathbb{R}^d \to \mathbb{R}$ is $\alpha$-strongly convex and $L$-smooth if at any point $x\in \R^d$, its Hessian matrix is positive definite, and its eigenvalues are at least $\alpha$ and at most $L$:
$$\alpha I \preceq \nabla^2\varphi(x) \preceq L I, \quad \forall x\in \R^d,$$ where $A\preceq B$ signifies that the difference $B-A$ between two matrices is positive semi-definite.

Consider the following sets of potential functions:
$$ \Phi_\alpha:=\{\varphi: \text{ such that }\alpha I \preceq \nabla^2\varphi(x) \preceq L I \} \quad \text{ for } L>\alpha\geq 0,$$
$$\Phi:= \cup_{\alpha>0} \Phi_\alpha.$$
Throughout the paper, we suppose $L$ is fixed and known.
Next, we define the following set of optimal maps:
$$\mathcal{T}:=\{T:\Omega \to \Omega : T= \nabla \varphi, \text{ for some } \varphi \in \Phi \}.$$

\noindent We will require the following regularity of the optimal maps involved in the model, in order to ensure idenfitiability:
\begin{assumption}{\label{assumptionT}}
	The map $T_0$ belongs to the class  $\mathcal{T}$.
\end{assumption}

\begin{assumption}{\label{noise_map}}
	The maps $T_{\epsilon_i}$ are i.i.d random elements in
$\mathcal{T}$,  with $\E (T_{\epsilon_i}) = \id$.
\end{assumption}

 \noindent With these assumptions and definitions in place, we can now establish identifiability:

\begin{theorem}{\label{identifiability-d}}
(Identifiability)
Assume that the law $P$ induced by model \eqref{model-d} satisfies Assumptions \ref{assumptionMu}, \ref{assumptionT}, and \ref{noise_map}. Then, the regressor operator $\Gamma(\mu)=T_0\#\mu$ in the model \eqref{model-d} is identifiable over the class of maps $T \in \mathcal{T}$, up to $Q$-null sets. Specifically,  for any map $T \in \mathcal{T}$ such that $\norm{T-T_0}_{L^2(Q)}>0$, it holds that
$$M(T)>M(T_0),$$  
where for any $T\in \mathcal{T}$,
\begin{equation}{\label{population-functional-d}}
 M(T):= \frac{1}{2}\int_{\mathcal{W}_2(\Omega)\times \mathcal{W}_2(\Omega)} d^2_{\mathcal{W}}(T\# \mu,\nu) \diff P(\mu,\nu).
 \end{equation}

\end{theorem}

\begin{remark}
	[Identifiability $Q$-almost everywhere]
	Theorem \ref{identifiability-d} establishes that $T_0$ is identifiable, up to $Q$-null sets, and this holds true under minimal conditions on the input measures $\mu$. The intuition behind this form of identifiability is similar to the one provided in Remark 1 of  \citet{ghodrati2022distribution} for the case of one dimension: 
	If the measure $Q$ is supported on a subset $\Omega_0\subset \Omega$, the map $T_0$ cannot be identified on $\Omega\setminus \Omega_0$. But, if the measure $Q$ is mutually absolutely continuous with the Lebesgue measure, then identifiability is also true almost everywhere on $\Omega$ with respect to Lebesgue measure.
	
	This equivalence can be achieved by enforcing additional conditions on the law of random covariate measures $\mu$. One straightforward condition is to require that the input measures, $\mu$, have a bounded density from below with positive probability. However, this implies that the support of $\mu$ equals $\Omega$ with positive probability, which may be restrictive since we want our model to include scenarios where none of the covariate measures have the full support on $\Omega$.
	
	A weaker condition to ensure the equivalence of $Q$ with the Lebesgue measure is to assume the existence of a cover $\{E_m\}_{m\geq 1}$ of $\Omega$ such that the probability $P_M\{E_m\subseteq\mathrm{supp}(f_\mu)\}>0$ is greater than zero for all $m$. This condition suggests that different covariate measures can provide information about $T_0$ on different subsets of $\Omega$, but collectively, they must provide information about all of $\Omega$. Consider an example where $\Omega$ is the $d$-dimensional unit cube  and $\mu$ is defined as the normalized Lebesgue measure on $S=\big([U_1,U_1+1/3] \mod 1 \big)\times\cdots\times \big([U_d,U_d+1/3] \mod 1 \big)$. Here, $\{U_i\}_{i=1}^d$ are independent uniform random variables on $[0,1]$. In this scenario, none of the $\mu$ realizations are supported on $\Omega$, yet the ``cover condition" is met. This remark directly extends the analogous observations in the one-dimensional case.
\end{remark}

Figure \ref{fig:flow} illustrates the output of Model \eqref{model-d} when $d=2$. In the first plot of each column, blue dots represent samples from a covariate distribution $\mu$.  The black dots are sampled from the (conditional Fr\'echet mean) distribution $T_0\#\mu$. The displacement vector field generated by $T_0-\id$ is depicted with flow curves, with the colour intensity along these curves reflecting the magnitude of displacement at each point. We then examine 4 different random maps $T_\epsilon$. In the next 4 plots, we observe samples from the response distribution $\nu=T_\epsilon\#T_0\#\mu$, represented by red dots. In each plot, the flow curves represent the displacement vector field $T_\epsilon-\id$.

\begin{figure}
	\centering
	\includegraphics[width=0.866\linewidth]{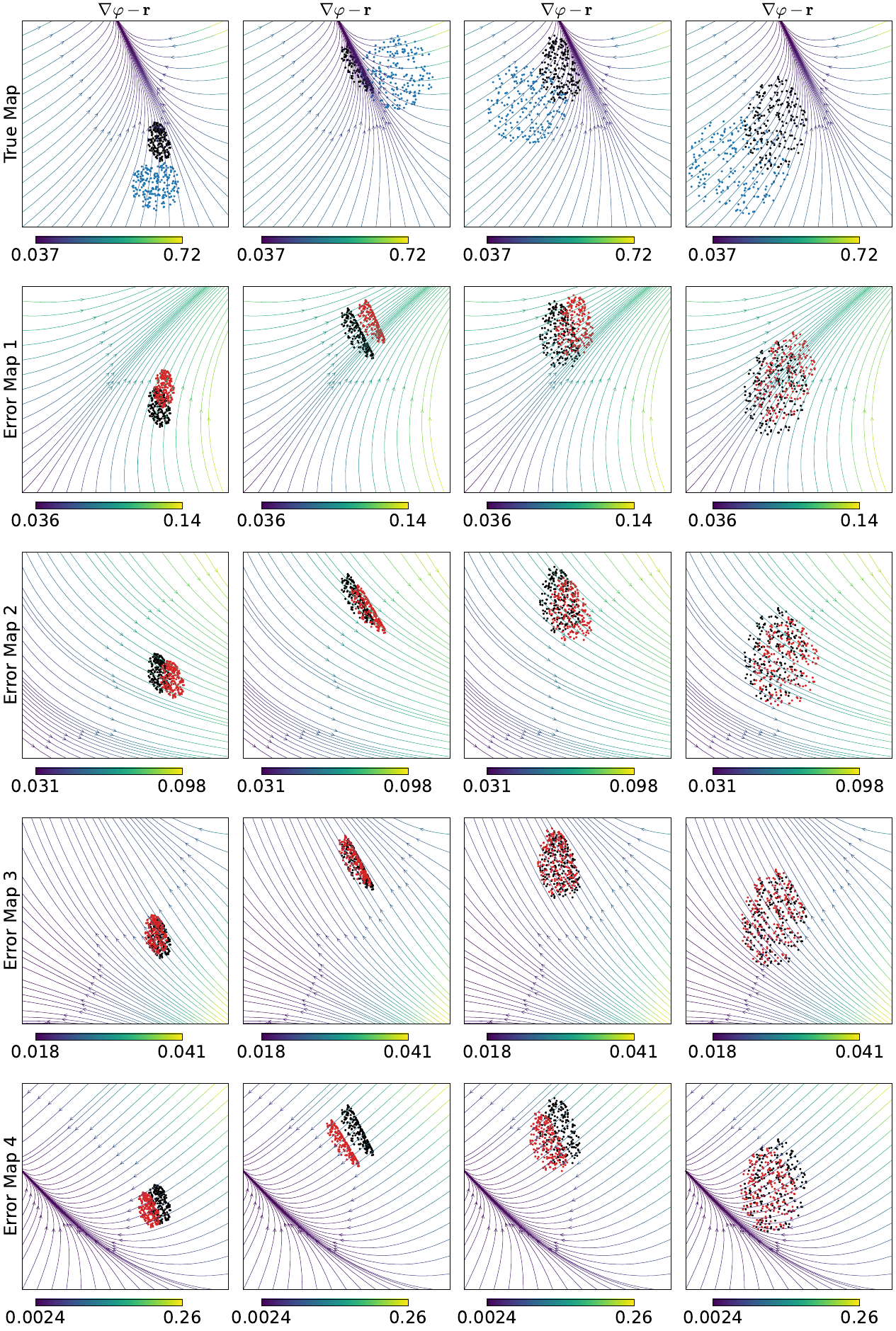}
	\caption{Illustration of Model \eqref{model-d} for $d=2$, showing samples from $\mu$ (blue), $T_0\#\mu$ (black), and $\nu=T_\epsilon\#T_0\#\mu$ (red) for four different realisations of the error map, along with corresponding displacement vector field (flow curves).}
	\label{fig:flow}
\end{figure}

\section{Statistical Analysis}{\label{Statistical Analysis}}

To obtain a consistent estimator and derive its rate of convergence, we use empirical process theory. This requires us to make additional regularity assumptions on the model. We consider the case where the true map $T_0$ satisfies a H\"older condition (defined below).

\begin{assumption}{\label{assumption_compact}}
	The domain $\Omega$ is a compact subset of $\R^d$.
\end{assumption}

\begin{definition}(H\"older Space)
	For any vector $\bk \in \mathbb{N}^d$ with coordinates $(k_1,\cdots,k_d)$ write $|\bk|=\sum_{i=1}^d k_i$ and define the differential operator
	$$D^{\bk}=\frac{\partial^{|\bk|}}{\partial x_1^{k_1}\cdots \partial x_d^{k_d}}.$$

	For  any  real  number $\beta >0$,  we  define  the  H\"older  norm  of  smoothness $\beta$ of a $\floor{\beta}$-times differentiable function $f:\Omega\to\R$ as 
	\begin{equation}
		\norm{f}_{C^\beta}:= \max_{|\bk|\leq \floor{\beta}}\sup_x |D^{\bk} f(x)|+ \max_{|\bk|=\floor{\beta}} \sup_{x \neq y} \frac{ |D^{\bk} f(x)-D^{\bk} f(y)|}{\norm{x-y}^{\beta-\floor{\beta}}}.
	\end{equation}
	The H\"older ball of smoothness $\beta$ and radius $L >0$, denoted by $C^{\beta}_R(\Omega)$, is then defined as the class of $\floor{\beta}$-times continuously differentiable functions with H\"older norm bounded by the radius $L$:
	$$C^{\beta}_R(\Omega):=\{f \in C^{\floor{\beta}}(\Omega): \norm{f}_{C^\beta}\leq R \}.$$
\end{definition}
We might occasionally drop the argument $\Omega$ when the underlying space can be understood from the context. Now consider the following sets of maps:
$$\mathcal{T}_{\beta,\gamma,R}:=\{T: T=\nabla \varphi, \varphi \in \Phi \cap \overline{C^{\beta+\gamma}_R}^{\norm{.}_{C^\beta}}\}$$
$$\mathcal{T}_{\beta,R}:=\{T: T=\nabla \varphi, \varphi \in \Phi_0 \cap C^{\beta}_R\}.$$

\noindent  It holds that $\mathcal{T}_{\beta,\gamma,R} \subset \mathcal{T}_{\beta,3R}$ (by way of Lemma \ref{embedding} in Section \ref{sec:proofs}). Our assumption now is:

\begin{assumption}{\label{assumptionT_beta_holder}}
The map $T_0$ belongs to the set $\mathcal{T}_{\beta,\gamma,R}$, where $\beta$, $\gamma$, and $R$ are positive constants and $\floor{\beta+\gamma} = \floor{\beta}$.
\end{assumption}

\begin{remark}
The $\beta$-H\"older function classes appear frequently in optimization and such regularity assumption is standard in non-parametric regression. In statistical optimal transport theory (see e.g. \cite{gunsilius2022convergence,hutter2019minimax}) similar assumptions are imposed to estimate optimal transport maps when observing samples from distributions.
\end{remark}

Our objective is to obtain an estimator, denoted by $\hat{T}_{N,(\beta,\gamma,R)}$, of the unknown map $T_0$. To do so, we define $\hat{T}_{N,(\beta,\gamma,R)}$ as the constrained minimizer of the sample version of the functional $M$,

\begin{equation}{\label{estimator_functional}}
\hat{T}_{N,(\beta,\gamma,R)}:=\argmin_{T \in \mathcal{T}_{\beta,\gamma,R}} M_N(T), \quad \quad M_N(T):= \frac{1}{2N} \sum_{i=1}^N d^2_{\mathcal{W}}(T\# \mu_i,\nu_i),
\end{equation}
Here, $(\mu_i,\nu_i)$ are independent samples drawn from $P$ for $i=1,\dots,N$. The constraint amounts to the requirement that $T \in \mathcal{T}_{\beta,\gamma,R}$. The minimizer might not be unique. We can take any of the minimizers of \eqref{estimator_functional} as $\hat{T}_{N,(\beta,\gamma,R)}$.  But it will exist in $\mathcal{T}_{\beta,3R}$, provided $\beta>2$:

\begin{theorem}{\label{existence_estimator}}
	The minimization problem \eqref{estimator_functional} has a solution in $\mathcal{T}_{\beta,3R}$ when $\beta>2$. 
\end{theorem}

\begin{remark}
	The functional $d^2_{\mathcal{W}}(T\# \mu,\nu)$ is not necessarily convex with respect to $T$, meaning that we \textit{cannot} expect $d^2_{\mathcal{W}}([a T_1 + (1-a)T_2]\# \mu,\nu)\leq a d^2_{\mathcal{W}}(T_1\# \mu,\nu)+(1-a) d^2_{\mathcal{W}}(T_2\# \mu,\nu)$. The reason is that the Wasserstein distance is not necessarily convex with respect to geodesics. And, if we set $T_2$ as the identity function, $[a T_1 +(1-a)Id]\#\mu$ represent the geodesic between $\mu$ and $T_1\#\mu$. Under this scenario, selecting $T_1$, $\mu$ and $\nu$ as per example 9.1.5 of \cite{ambrosio2008gradient} results in the violation of the inequality.
	
	Nevertheless, when $\nu$ is close to $\mu$ we can show convexity. Denote by $T_{\mu\to \nu}$ the optimal transport from $\mu$ to $\nu$. 
	Suppose $T_{\mu\to \nu}$ is gradient of a $\lambda$-strongly convex function. Then based on Theorem 6 of \citet{manole2021plugin} and inequality \eqref{PC-d}, we have:
	
	\begin{equation}\label{thm6}
		\begin{split}
		d^2_{\mathcal{W}}([a T_1 + (1-a)T_2]\# \mu,\nu) &\leq \norm{a T_1 +(1-a) T_2 -T_{\mu \to \nu}}^2_{L^2(\mu)} \\
		&\leq \lambda^2  d^2_{\mathcal{W}}([a T_1 + (1-a)T_2]\# \mu,\nu),
		\end{split}
	\end{equation}
	and similarly
	\begin{equation}\label{thm62}
		\begin{split}
		a d^2_{\mathcal{W}}(T_1 \# \mu,\nu)+(1-a) d^2_{\mathcal{W}}(T_2 \# \mu,\nu)  &\leq a \norm{T_1-T_{\mu \to \nu}}^2_{L^2(\mu)}+(1-a) \norm{T_2-T_{\mu \to \nu}}^2_{L^2(\mu)} \\
		&\leq a \lambda^2  d^2_{\mathcal{W}}(T_1\# \mu,\nu)+ (1-a) \lambda^2  d^2_{\mathcal{W}}(T_2\# \mu,\nu).
		\end{split}
	\end{equation}
	Also, by strict convexity of squared norm we have 
\begin{equation}
\begin{split}
\norm{a T_1 +(1-a) T_2 -T_{\mu \to \nu}}^2_{L^2(\mu)}< a \norm{T_1-T_{\mu \to \nu}}^2_{L^2(\mu)} +  (1-a) \norm{T_2-T_{\mu \to \nu}}^2_{L^2(\mu)}.
\end{split}
\end{equation}
	
As we observe from the above equation, the middle value in inequality \eqref{thm6} is strictly smaller than the middle value in inequality \eqref{thm62}. Moreover, if $\lambda$ approaches 1, both inequalities' right-hand sides converge to their respective left-hand sides. Consequently, for small values of $\lambda$, the left-hand side of \eqref{thm6} is majorised by the left-hand side of \eqref{thm62}, which establishes convexity. It is worth mentioning that when $\lambda$ approaches 1, $T_{\mu\to\nu}$ converges to the identity map, and $\nu$ converges to $\mu$, providing a perspective on why convexity occurs when $\nu$ is close to $\mu$.
	
\end{remark}

To evaluate the estimator's quality, we will use the \textit{Fr\'echet mean squared error}, which is  the natural risk function in this context:
$$R(T):=\E_{\mu \sim P_M} d^2_{\mathcal{W}}(T_0\#\mu,T\#\mu)=\int d^2_{\mathcal{W}}(T_0\#\mu,T\#\mu) \diff P_M(\mu).$$

The value on the right-hand side defines a semi-metric on the set of maps $\mathcal{T}$. More specifically, for any two maps $T_1,T_2 \in \mathcal{T}$, we can define

$$\rho(T_1,T_2) = \Big[\int d^2_{\mathcal{W}}(T_1\#\mu,T_2\#\mu)\diff P_M(\mu)\Big]^{1/2}.$$
\begin{lemma}[Semi-metric property of $\rho$]{\label{semi-metric}}
	The map  $\rho(\cdot,\cdot)$ satisfies all the properties of a metric, except that there may exist pairs $T_1\neq T_2$ such that $\rho(T_1,T_2)=0$. 
\end{lemma}

\vspace{3mm}

\noindent Since $\mathcal{W}_2(\R^d)$ is non-negatively curved (equation \eqref{PC-d}), for any two maps $T_1,T_2 \in \mathcal{T}$ we have
\begin{equation}{\label{norm_comparison}}
\begin{split}
\rho^2(T_1,T_2)&=\int d^2_{\mathcal{W}}(T_1\# \mu,T_2\#\mu) \diff P_M(\mu)\\ &\leq \int \norm{T_1-T_2}^2_{L^2(\mu)} \diff P_M(\mu) \\
&=\int \int |(T_1-T_2)(x)|^2 \diff \mu(x) \diff P_M(\mu) \\
&=\norm{T_1-T_2}_{L^2(Q)}^2,
\end{split}
\end{equation}
with equality when $d=1$.
We use this to derive an upper bound for the rate of convergence of $\hat{T}_N$ with respect to semi-metric $\rho$.
\begin{theorem}(Rate of Convergence){\label{rate_of_convergence-d}}
    Suppose the Assumptions \ref{assumptionMu}, \ref{assumptionT}, \ref{noise_map}, \ref{assumption_compact} and \ref{assumptionT_beta_holder} are satisfied with some $\beta>2$ and $\gamma,R>0$. Then
    $$r_N\rho(\hat{T}_{N,(\beta,\gamma,R)},T_0) = O_\mathbb{P}(1),$$
	where
	\begin{equation}
		r_N=
		\begin{cases} 
			N^{\frac{\beta}{2\beta+d}} & \text{if } \beta>\frac{d}{2}, \\[2ex]
			\frac{N^{1/4}}{(\log(N))^{1/2}} & \text{if } \beta = \frac{d}{2},\\[2ex]
			N^{\frac{1}{\beta(2\beta+d)}} & \text{if } \beta < \frac{d}{2}.
		\end{cases}
	\end{equation}
 In particular, for $d>4$, and when $\frac{d}{2}<\beta<\infty$, the rate is between $N^{-1/4}$ and $N^{-1/2}$.
\end{theorem}

\begin{remark}[The case $d=1$]
	In the 1-dimensional case, \citet{ghodrati2022distribution} achieved a convergence rate of $N^{-1/3}$ using the $L^2(Q)$-norm, and \citet{ghodrati2023minimax} demonstrated that this rate is minimax optimal. It's worth noting that, for $d=1$, the $L^2(Q)$-norm is equivalent to the semi-metric $\rho$ under the assumption that $T_0$ is a non-decreasing map. As a result, we can compare their rates with the one provided by Theorem \ref{rate_of_convergence-d}. When $d=1$, Theorem \ref{rate_of_convergence-d} suggests that when $\frac{d}{2}<\beta\leq \infty$, the convergence rate lies between $N^{-2/5}$ and $N^{-1/2}$. The faster convergence rate is a result of the additional smoothness assumption, and thus, there is no inconsistency between the two results.
\end{remark}

\section{Computation}{\label{Computation}}

In this section, we propose a method to compute the estimator $\hat{T}_{N,(\beta,\gamma,R)}$ defined in \eqref{estimator_functional}. We first show $M_N$ can be expressed as a difference of two convex functions with respect to the domain $\mathcal{T}$. Then, after showing some results about Gateaux differentiability of $M_N$, we propose a difference of convex functions algorithm (DCA) to minimize $M_N$.

	

\begin{lemma}{\label{DC}}
Consider the functional $M_N(T)$. It can be expressed as the difference of two convex functionals, namely $g(T)$ and $h(T)$, where
\begin{align*}
	g(T)&:=\sum_{i=1}^N \int \norm{T(x)}^2 \diff \mu_i(x) + \int \norm{y}^2 \diff \nu_i(y),\\
	h(T)&:=2\sum_{i=1}^N \sup_{\gamma_i \in \Gamma(\mu_i,\nu_i)} \int \langle T(x),y \rangle \diff \gamma_i(x,y).
\end{align*}
\end{lemma}

\begin{lemma}{\label{derivative-d}}
	For any $\alpha>0$, the functionals $M$ and $M_N$ are Gateaux-differentiable at any maps in $\mathcal{T}$ and
	there exist couplings $\gamma_{\mu,\nu}\in \Gamma(\mu,\nu)$ and $\gamma_{\mu_i,\nu_i}\in \Gamma(\mu_i,\nu_i)$ such that:
	
	\begin{equation}{\label{D_M}}
	\begin{split}
	&D_\eta M(T)=\int\int \langle \eta(x),T(x)-y\rangle\diff \gamma_{\mu,\nu}(x,y) \diff P(\mu,\nu),\\   &d^2_{\mathcal{W}}(T\#\mu,\nu)=\int  \norm{T(x)-y}^2 \diff \gamma_{\mu,\nu}(x,y)
	\end{split}
	\end{equation}
	and
	\begin{equation}{\label{D_M_N}}
	\begin{split}
	& D_\eta M_N(T)=\frac{1}{N}\sum_{i=1}^N \int \langle\eta(x),T(x)-y\rangle\diff \gamma_{\mu_i,\nu_i}(x,y),\\
	&d^2_{\mathcal{W}}(T\#\mu_i,\nu_i)=\int  \norm{T(x)-y}^2 \diff \gamma_{\mu_i,\nu_i}(x,y).
	\end{split}
	\end{equation}
\end{lemma}

\vspace{3mm}

\begin{lemma}{\label{derivative-h}}
	The Gateaux derivative of the functional 
	$h(T)$ defined in Lemma \ref{DC},
	at any map $T\in \mathcal{T}$, in a direction $\eta$, is given by
	\begin{equation*}
		D_\eta h(T) =  2\sum_{i=1}^N \int \langle \eta(x),y \rangle \diff \gamma_i^T(x,y),
	\end{equation*}
	where $\gamma_i^T \in \Gamma(\mu_i,\nu_i)$ is chosen such that for each $i\in\{1,\cdots,N\}$, 
	$$d^2_{\mathcal{W}}(T\#\mu_i,\nu_i)=\int  \norm{T(x)-y}^2 \diff \gamma_i^T(x,y).$$
	Furthermore, if for each $i\in\{1,\cdots,N\}$,there exist an optimal map $S_i$ such that $S_i\#T\#\mu_i=\nu_i$, then $\gamma_i^T=(I, S_i\circ T)\#\mu$, implying
	\begin{equation}{\label{D_h}}
		D_\eta h(T) =  2\sum_{i=1}^N\int \langle \eta(x),S_i\circ T(x) \rangle \diff \mu_i(x).
	\end{equation}
\end{lemma}

\begin{algorithm}[H]
\caption{DCA for Minimizing \( g(T) - h(T) \)}
\begin{algorithmic}[1]
    \State Initialize \( T^0 \) and set \( k = 0 \).
    \Repeat
        \State Compute a gradient \( D_\eta h(T^k) \).
        \State Solve the convex subproblem:
        $T^{k+1} = \arg\min_T g(T) - D_T h(T^k) $
        \State Check for convergence: if $\norm{T^{k+1}-T^k}$ is small enough, or if the improvement in the objective is below a threshold, or if $k$ reaches a maximum number of iterations, then stop.
        \State Update \( k \leftarrow k + 1 \).
    \Until{convergence criterion is met}
\end{algorithmic}
\end{algorithm}

\begin{remark}
	The iterative update process for $T^{k+1}$ can be expressed as
	\begin{equation*}
		\begin{split}
			T^{k+1} &= \arg\min_T \{ g(T) - D_T h(T^k) \} \\
			&= \arg\min_T \{ \sum_{i=1}^N \int \|T(x)\|^2 d\mu_i(x) - 2\int \langle T(x), y \rangle d\gamma_i^{T^k}(x, y) \} \\
			&= \arg\min_T \{ \sum_{i=1}^N \int \|T(x) - y\|^2 d\gamma_i^{T^k}(x, y) \} \\
			&= \arg\min_T \{ \sum_{i=1}^N  \int \|T(x) - T_{T^k \# \mu_i \to \nu_i} \circ T^k(x)\|^2 d\mu_i(x) \}
		\end{split}
	\end{equation*}
	Notably, in the one-dimensional scenario, the expression $T_{T^k \# \mu_i \to \nu_i} \circ T^k(x)$ simplifies to $T_{\mu_i \to \nu_i}(x)$.This simplification reveals that the update rule converges in a single iteration, aligning with the formulation presented in equation (6) of \citet{ghodrati2022distribution}.
\end{remark}

\section{Gaussian Case}{\label{Gaussian Case}}

In this section, we focus on a particular instantiation of \eqref{model-d} where both covariate and response measures are non-degenerate centred Gaussian measures on $\R^d$.

In the Gaussian scenario, the Wasserstein distance between two such measures, $\mu = N(0,M)$ and $\nu = N(0,N)$, is expressed as
\begin{equation*}
d^2_{\mathcal{W}}(\mu,\nu) = \text{tr}(M) + \text{tr}(N) - 2\text{tr}\left(\left(N^{1/2}M N^{1/2}\right)^{1/2}\right).
\end{equation*}
Furthermore, when $M$ is injective, the optimal transport map from $\mu$ to $\nu$ can be represented by a $d \times d$ positive definite matrix in closed form:
\begin{equation*}
T_{\mu}^\nu = M^{-1/2} \left(M^{1/2} N M^{1/2}\right)^{1/2} M^{-1/2},
\end{equation*}
which equivalently satisfies
\begin{equation*}
N = T_{\mu}^\nu M T_{\mu}^\nu.
\end{equation*}
Consequently, for Gaussian distributions, the model \eqref{model-d} assumes the form
\begin{equation}
\label{model-gaussian}
N_i = Q_{\epsilon_i}TM_iTQ_{\epsilon_i},
\end{equation}

Define the set of matrices
\begin{equation*}
	S^d_{++}:=\{M\in \R^{d\times d} : \forall x\in \R^d \quad 0<x^\top M x \leq L \}, \quad \text{for } L>0.
\end{equation*}

\noindent The assumptions \ref{assumptionMu}, \ref{assumptionT}, and \ref{noise_map} are now reformulated as follows:
\begin{assumption}{\label{assumption-gaussian}}
The input matrices $M_i$ and the matrix $T$ are symmetric, positive definite with bounded eigen values belonging to the set $S^d_{++}$. Additionally, the matrices $Q_{\epsilon_i}$ are independent random positive definite matrices in $S^d_{++}$ with expected value $\E(Q_{\epsilon_i})=I$.
\end{assumption}

\noindent Given these formulations, the functional $M_N$ takes the form
\begin{equation}
M_N(T) = \sum_{i=1}^N \left[ \text{tr}(TM_iT) + \text{tr}(N_i) - 2\text{tr}\left(\left(N_i^{1/2}TM_i TN_i^{1/2}\right)^{1/2}\right) \right].
\end{equation}
And the estimator $\hat{T}_N$ is defined as the minimizer of $M_N$ within $S^d_{++}$.

The statistical analysis in Section \ref{Statistical Analysis}, is limited to the case where the predictor and response distributions are compactly supported. Here, we extend the analysis to the case where the predictor distributions are Gaussian and supported on $\mathbb{R}^d$.

Leveraging Lemma \ref{derivative-d} and Lemma \ref{derivative-h}, the function $h(T)$ is defined as
\begin{equation*}
h(T) := \sum_{i=1}^N 2\text{tr}\left(\left(N_i^{1/2}TM_i TN_i^{1/2}\right)^{1/2}\right).
\end{equation*}

\begin{lemma}{\label{derivative-gaussian}}
	The Gateaux derivative of $M_N$ and $h$ at $T \in S^d_{++}$ in the direction $\eta$ are
	\begin{equation*}
		D_\eta M_N(T) = \sum_{i=1}^N 2\text{tr}(\eta M_i (T-S_{i} T)),\quad \quad D_\eta h(T) = \sum_{i=1}^N 2\text{tr}(\eta M_i S_{i} T),
	\end{equation*}
	where $S_{i}$ denotes the optimal transport map between $T M_i T$ and $N_i$, formally given by
	\begin{equation*}
	S_{i} = G_{i}^{-1/2} \left(G_{i}^{1/2} N_i G_{i}^{1/2}\right)^{1/2} G_{i}^{-1/2}, \quad \text{where} \quad  G_{i} = T M_i T.
	\end{equation*}
\end{lemma}
To evaluate the performance of the estimator $\hat{T}_N$, we use the $\rho$-norm. Note that, in defining \(\rho\), we adopt a slight abuse of notation by using \(T\) to represent the optimal map, expressed as \(T(x) = Tx\).
\begin{theorem}(Rate of Convergence-Gaussian Case){\label{rate_of_convergence-gaussian}}
    Suppose the Assumption \ref{assumption-gaussian} is satisfied. Then
    $$N^{1/2} \rho(\hat{T}_N,T_0) = O_\mathbb{P}(1).$$
\end{theorem}

To identify the minimizer of $M_N$ in $S^d_{++}$, the Difference of Convex functions Algorithm (DCA) is employed, and the update rule is thus formulated as
$$T^{K+1} = \arg\min_{T \in S^d_{++}} \sum_{i=1}^N\left(tr(TM_iT)-D_T h(T^k)\right)=\arg\min_{T \in S^d_{++}} \sum_{i=1}^N\left(tr(TM_iT)-2tr(T M_i S_{k,i} T^{k}) \right),$$
where $S_{k,i}$ denotes the optimal transport map between $T^k M_i T^k$ and $N_i$.

\begin{algorithm}[H]\label{DCA-Gaussian}
    \caption{DCA for the Gaussian case}
    \begin{algorithmic}[1]
        \State Initialize \( T^0 \) and set \( k = 0 \).
        \Repeat
            \State Compute for each \( i \in \{1, \ldots, N\} \):
            \begin{itemize}
                \item[] \begin{enumerate}[a.]
                    \item \( G_{k,i} = T^{k} M_i T^{k} \)
                    \item \( S_{k,i} = G_{k,i}^{-1/2} \left( G_{k,i}^{1/2} N_i G_{k,i}^{1/2} \right)^{1/2} G_{k,i}^{-1/2} \)
                \end{enumerate}
            \end{itemize}
            \State Compute \( D_\eta h(T^k) = \sum_{i=1}^N 2\,\text{tr}(\eta M_i S_{k,i} T^{k}) \).
            \State Solve the convex subproblem:
            \begin{align*}
                T^{k+1} &= \arg\min_{T \in S^d_{++}} \Bigg\{ \sum_{i=1}^N \Big( \text{tr}(T M_i T) - 2\,\text{tr}(T M_i S_{k,i} T^{k}) \Big) \Bigg\}
            \end{align*}
            \State Update \( k \leftarrow k + 1 \).
        \Until{\( k \) reaches a maximum number of iterations.}
    \end{algorithmic}
\end{algorithm}

In this case, the DCA provably converges to a global minimizer see \citet[Theorem 1]{yao2023globally}. 

\section{Simulation Study}\label{simulation}

In this section, we conduct a simulation study for the Gaussian case to illustrate the finite sample performance of the method. First, we generate random positive definite matrices \( M_i \). This involves creating an arbitrary random diagonal matrix \( D_i \) with positive entries and a random orthonormal matrix \( U_i \) of the same size, and then computing \( M_i = U_i D_i U_i^\top \). The orthonormal matrix \( U_i \) is sampled from the Haar measure. Additionally, to generate random positive definite matrices \( Q_i \) with a mean identity, we apply the same method, but for the matrices \( D_i \), we use entries drawn from a uniform distribution over [0.5, 1.5]. This ensures all entries are positive and have a mean of 1.

We then use the model \eqref{model-gaussian} to generate matrices \( N_i \). For each experiment, we also set \( T_0 \) as a randomly generated positive definite matrix, which remains fixed throughout. We generated \( N \) pairs of Gaussian distributions in \( \R^d \) using their covariance matrices \( M_i \) and \( N_i \) (refer to Figure \ref{fig:surfaces} for a model illustration). Next, we use Algorithm 2 (DCA for the Gaussian case) to compute the estimator $\hat{T}_N$ for the matrix \( T_0 \). The true map and estimated map for \( N=100 \) in the case of \( d=2 \) are shown in Figure \ref{fig:image1} and Figure \ref{fig:image2}.

To evaluate the effects of the number of pairs of distributions \( N \) and the dimension \( d \) on the estimator error, we considered \( d \) in \{2,3,10,20\} and \( N \) in \{4,8,16,32,64,128,256,512,1024,2048\}. For each pair \( (d, N) \), we generated 50 independent samples of \( N \) pairs of Gaussian distributions in \( \R^d \) using their covariance matrices \( M_i \) and \( N_i \), and computed the estimator \( \hat{T}_N \) for each sample. The logarithm of the squared error between the estimator and the true matrix \( T_0 \) is displayed in box plots in Figure \ref{fig:convergence_rate_Gaussian}. Observing that the medians of the box plots approximately align along a line with a slope of -1/2, we confirm that the rate of convergence is \( N^{-1/2} \) for all dimensions, as predicted by Theorem \ref{rate_of_convergence-gaussian}.

\begin{figure}[H]
	\centering
	\includegraphics[width=0.81\linewidth]{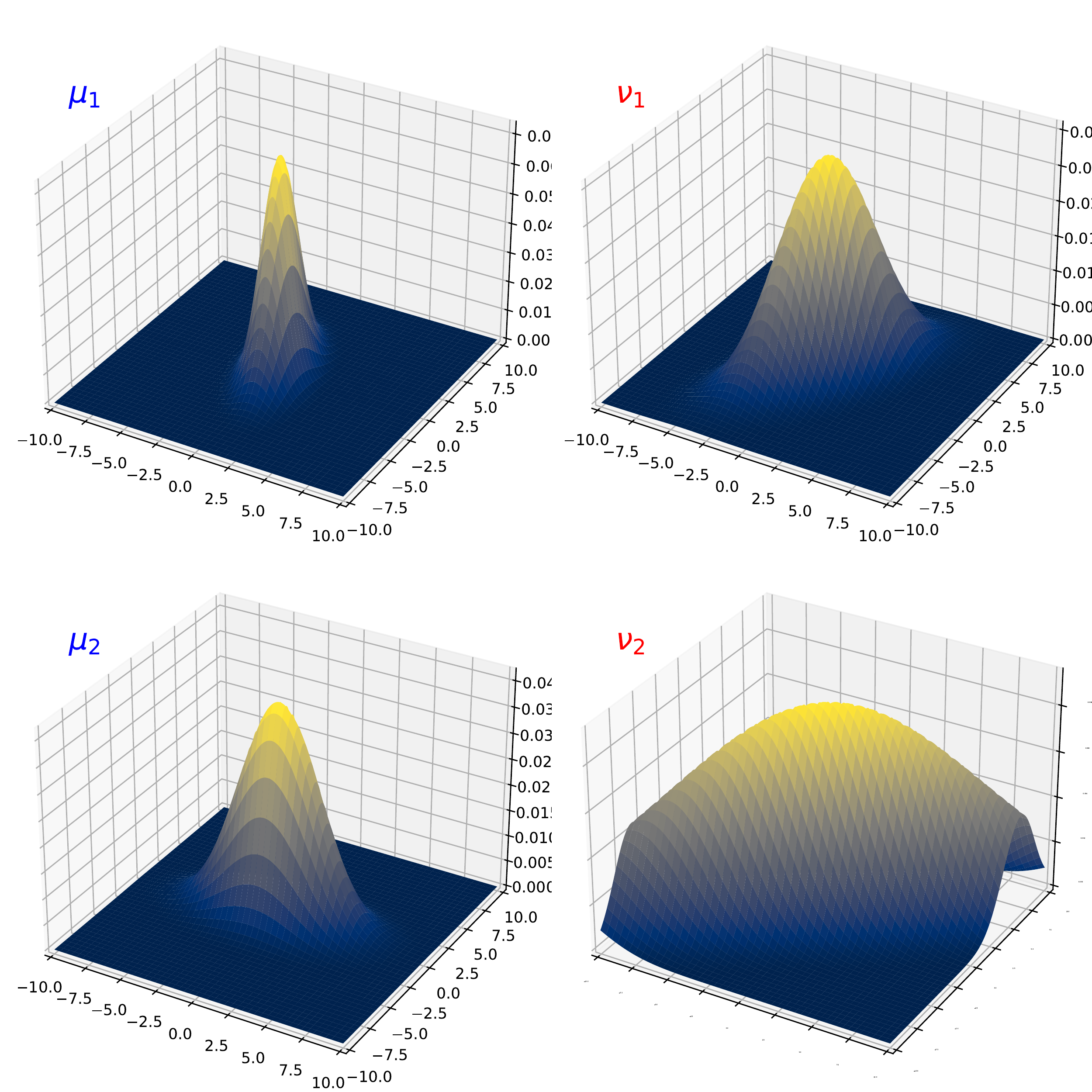}
	\caption{Illustration of Model \eqref{model-gaussian} for $d=2$, showing }
	\label{fig:surfaces}
\end{figure}

\begin{figure}[H]
	\centering
	\begin{minipage}{.5\textwidth}
	  \centering
	  \includegraphics[width=\linewidth]{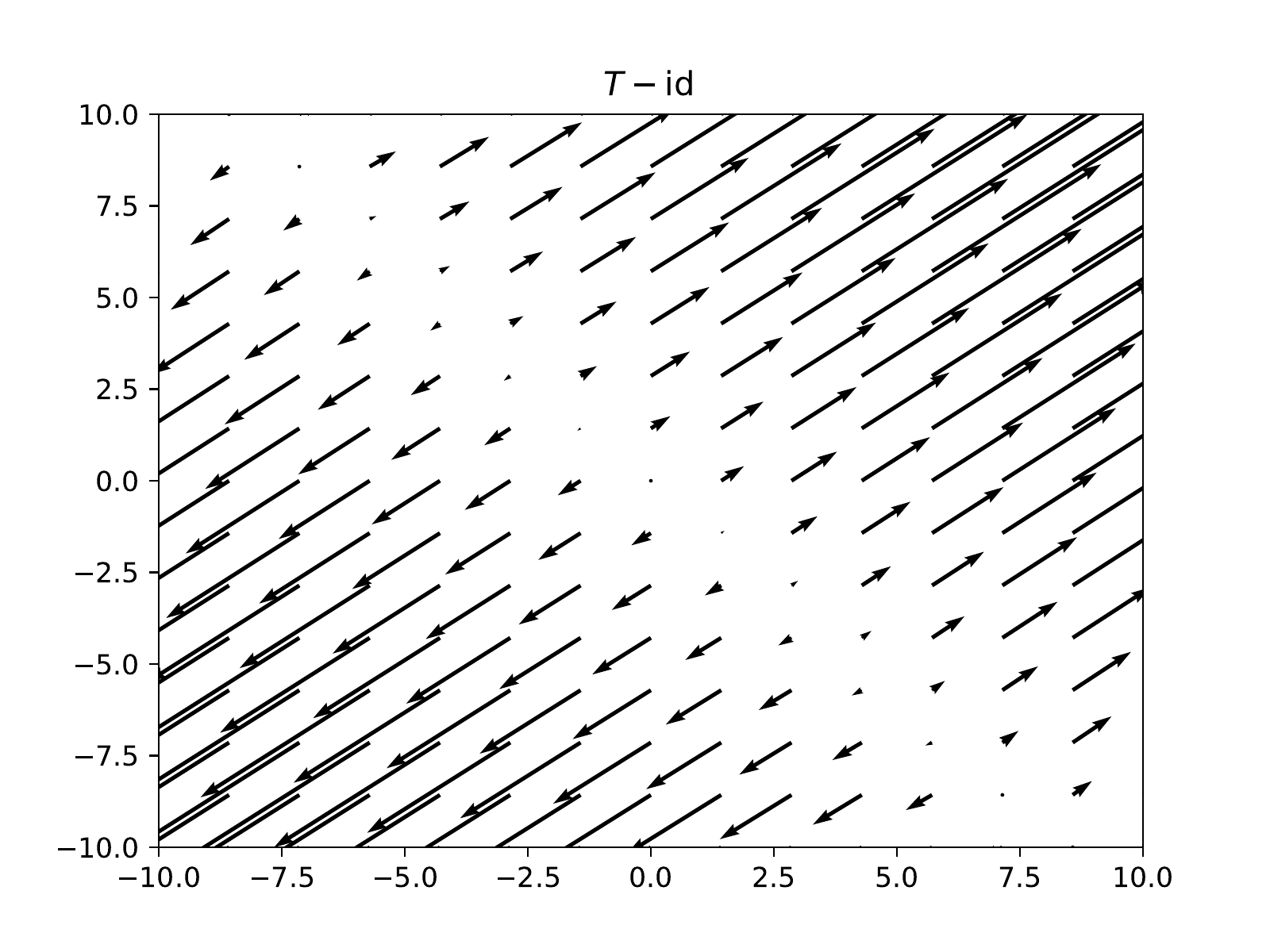}
	  \caption{True vector field}
	  \label{fig:image1}
	\end{minipage}%
	\begin{minipage}{.5\textwidth}
	  \centering
	  \includegraphics[width=\linewidth]{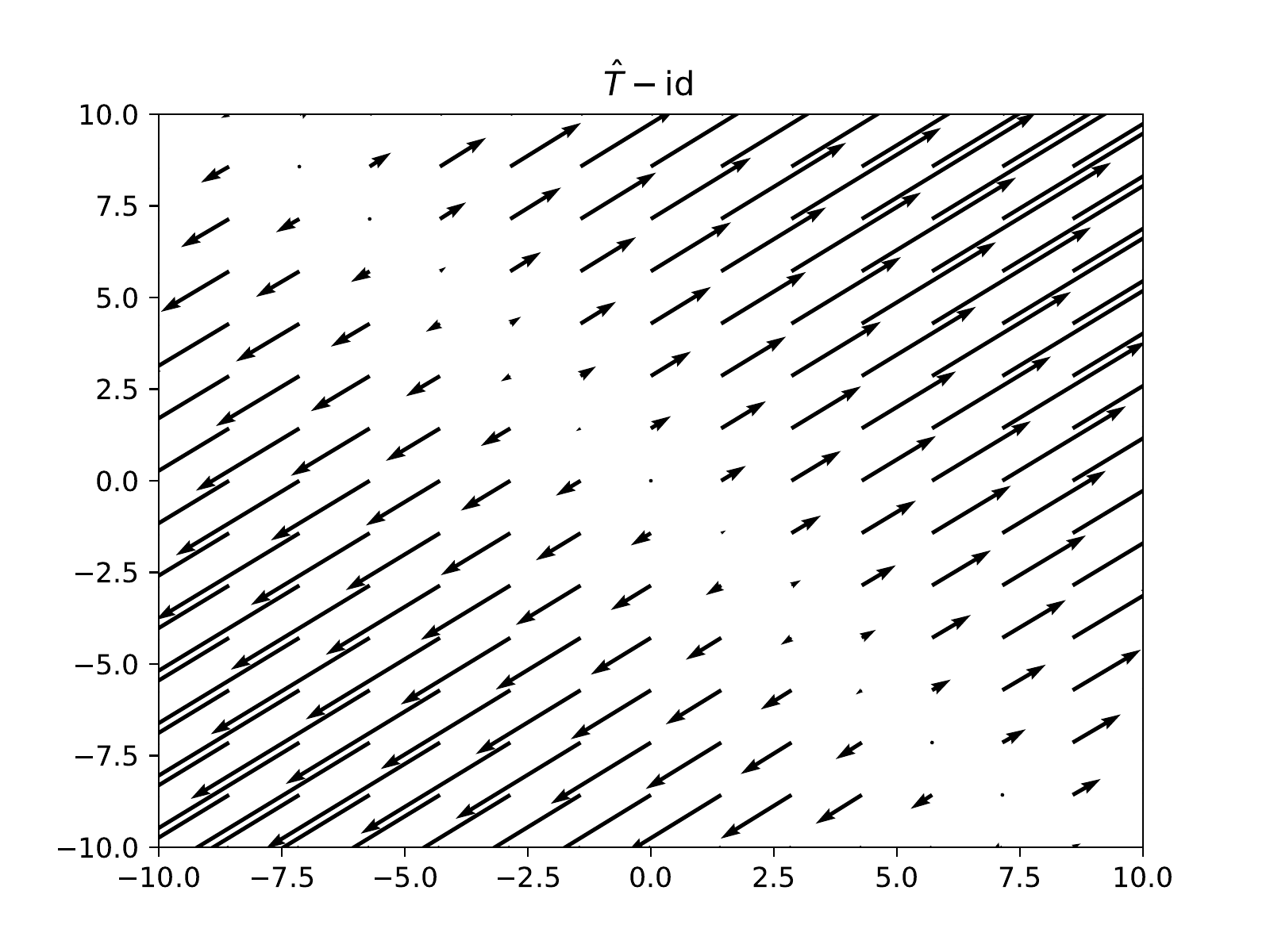}
	  \caption{Estimated vector field}
	  \label{fig:image2}
	\end{minipage}
  \end{figure}

  \begin{figure}[H]
	\centering
	\includegraphics[width=\linewidth]{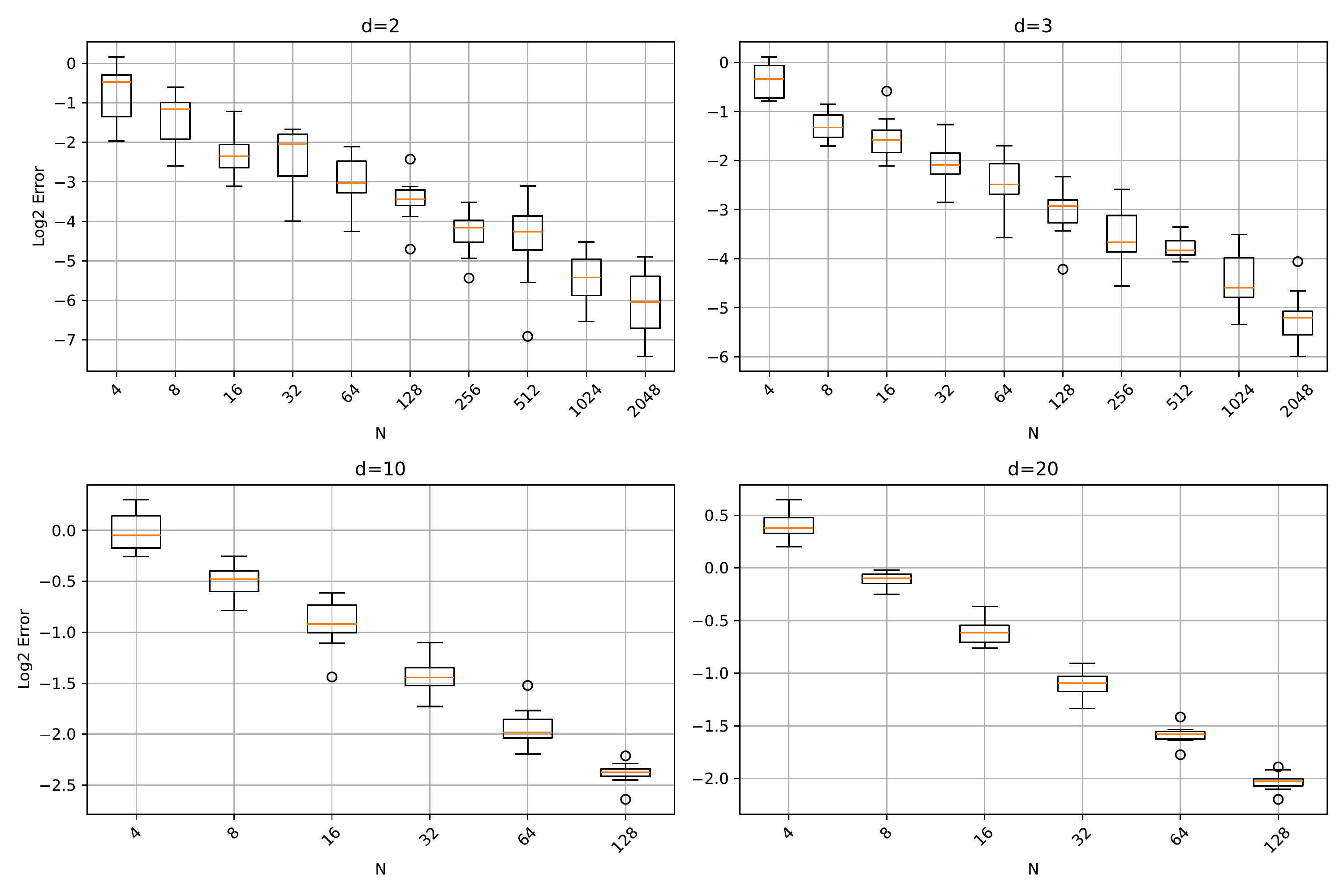}
	\caption{
	 Box plots showing the logarithmic squared error between the estimated matrix $\hat{T}_N$ and the true matrix $T_0$ based on 50 replication for various combination of $d$ and $N$. Given that the number of pairs $N$ are exponents of $2$ for each $d$, the median alignment along a -1/2 slope illustrates the $N^{-1/2}$ convergence rate, in accordance with Theorem \ref{rate_of_convergence-gaussian}.}
	\label{fig:convergence_rate_Gaussian}
\end{figure}

\section*{Acknowledgement}
We wish to thank Ramzi Sofiane Dakhmouche for discussions and comments regarding the DCA algorithm.

\section{Proofs}\label{sec:proofs}

In the next statement, we restate a Theorem from \citet{ponomarev1987submersions} that will be used to prove Lemma \ref{abs_pushforward}.

\begin{theorem}{\label{ponomarev}}
(Theorem 3 of \citet{ponomarev1987submersions})
When $f:\Omega\subset \R^d \to \R^d$ is continuous and almost everywhere differentiable, the following properties are equivalent:
\begin{itemize}
	\item $\mathrm{rank}\{f'(x)\}=d$ for almost all $x\in \Omega$,
	\item the $f$-preimage of any set of measure zero is set of measure zero, i.e. if $E\subset \R^d$ such that $\lambda(E)=0$ then $\lambda(f^{-1}(E))=0$.
\end{itemize}
\end{theorem}

\begin{lemma}{\label{abs_pushforward}}
$T\#\mu$ is absolutely continuous, when $\mu$ is absolutely continuous and $T\in \mathcal{T}$.
\end{lemma}
\begin{proof}
We need to show that if $A$ is a measurable set with Lebesgue measure $\lambda(A)=0$, then $T\#\mu(A) = 0$. We begin by noting that $T \in \mathcal{T}$ implies the Jacobian of $T$ is full rank. Using Theorem \ref{ponomarev}, 
we conclude that $\lambda(T^{-1}(A))=0$.
Next, recall that $T\#\mu(A) = \mu(T^{-1}(A))$. Since $\mu$ is absolutely continuous and $\lambda(T^{-1}(A))=0$, it follows that $\mu(T^{-1}(A))=0$. Therefore, we have $T\#\mu(A) = \mu(T^{-1}(A)) = 0$, as desired.
\end{proof}

Our argument for the identifiability of $T_0$ relies on a result from \citet{chewi2020gradient}, which was originally employed to establish quadratic growth of the Fr\'echet functional \eqref{frechet-functional} around its minimiser. In addition to identifiability, this finding is also applied to exhibit the quadratic growth of functional $M$ around its minimizer, crucial for determining the convergence rate of our proposed estimator. We will start by revisiting the notion of variance inequality, introduced by \citet{ahidar2020convergence}.A distribution $P$ conforms to a variance inequality with a positive constant $C_{var}$, if for any absolutely continuous measure $b\in \mathcal{W}_2(\mathbb{R}^d)$, the following inequality holds:
$$F(b)-F(b^*) \geq \frac{C_{var}}{2} d^2_{\mathcal{W}}(b,b^*),$$
where $b^*$ is the minimizer of $F$ defined by equation \eqref{frechet-functional}. Now we restate the following result from \citet{chewi2020gradient}:

\begin{theorem}\label{vineq}
	(Variance inequality, Theorem 6 of \citet{chewi2020gradient})
	Let $P$ be the law of a random measure in $\mathcal{W}_{2,\mathrm{ac}}(\Omega)$ with barycenter $b^*\in\mathcal{W}_{2,\mathrm{ac}}(\Omega)$. Assume there exists a measurable map $\varphi: \mathcal{W}_{2,\mathrm{ac}}(\Omega) \times \R^d \to \R$ such that for $P$-almost all $\mu$, $\varphi_\mu$ is an optimal Kantorovich potential for $b^*$ to $\mu$ and is $\alpha(T_{b^*\to \mu})$-strongly convex, where $T_{b^*\to \mu}=\nabla \varphi_\mu$ is the corresponding optimal map. Moreover, assume that for almost all $x\in \R^d$
	\begin{equation}{\label{potentials_integral}}
	\E_{\mu \sim P} [\varphi_\mu(x)] =\frac{1}{2}\norm{x}^2.
	\end{equation}
	Then, $P$ satisfies a variance inequality for all $b\in\mathcal{W}_{2,\mathrm{ac}}(\Omega)$ with constant
	$$C_{var}=\int \alpha(T_{b^*\to \mu}) \diff P(\mu).$$
\end{theorem} 

\begin{proof}[Proof of Theorem \ref{identifiability-d}]
	To prove that $T_0$ is the unique minimizer of the population functional in $\mathcal{T}$ up to $Q$-null sets, fix a measure $\mu_0$ in the support of $P_M$, and let $\nu$ be a random measure such that $\nu=T_{\epsilon}\#(T_0\#\mu_0)$ (where we recall that $\E(T_{\epsilon})=\id$). According to Lemma \ref{abs_pushforward}, $T_0 \# \mu_0$ is absolutely continuous, since $\mu_0$ is absolutely continuous, and $T_0\in \mathcal{T}_\alpha$. 
	
	Similarly, we can argue that $\nu$ is also absolutely continuous because $T_\epsilon \in \mathcal{T}$. Therefore according to Proposition 3.2.7 of \citet{panaretos2020invitation}, for any $\mu_0$, the induced random measure $\nu$ has a unique Fr\'echet mean. By Theorem 4.2.4 of \citet{panaretos2020invitation} we conclude this unique Fr\'echet mean is $T_0 \# \mu_0$. Therefore, for any $\mu_0$,
	
	$${\arg\inf}_{b\in \mathcal{W}_2(\Omega)} \int_{\mathcal{W}_2(\Omega)} d^2_{\mathcal{W}}(b,\nu) \diff P(\nu|\mu_0)=T_0\# \mu_0,$$
	where $P$ is the joint distribution of $(\mu,\nu)$ induced by Model \eqref{model-d}. Now we show that $T_0$ is a minimiser of $M$:
	
	\begin{equation*}
	\begin{split}
	M(T)&=\int d^2_{\mathcal{W}}(T\# \mu,\nu) \diff P(\mu,\nu)\\
	&=\int \int d^2_{\mathcal{W}}(T\#\mu_0,\nu)\diff P(\nu|\mu_0) \diff P(\mu_0)\\
	&\geq \int \int d^2_{\mathcal{W}}(T_0\#\mu_0,\nu)\diff P(\nu|\mu_0) \diff P(\mu_0)\\
	&=\int d^2_{\mathcal{W}}(T_0\# \mu,\nu) \diff P(\mu,\nu).
	\end{split}
	\end{equation*}
	We will show that $T_0$ is the unique minimizer up to $L^2(Q)$-norm. We apply Theorem \ref{vineq} in the following way: fix a measure $\mu_0$ in the support of $P_M$. As we stated, $T_0\#\mu_0$ is the conditional Fr\'echet mean of the random measure $\nu$ given $\mu_0$ (i.e. the baruycentre of the law $P(\nu|\mu_0)$). Define a mapping {$\varphi: \mathcal{W}_{2,\mathrm{ac}}(\Omega)\times \R^d \to \R$} such that $\varphi_{\nu|\mu_0}$ is equal to the Kantorovich potential of the optimal map $T_\epsilon$, where we abuse notation for tidiness and write $\nu|\mu_0\equiv T_\epsilon\#T_0\#\mu_0$ (the existence of such map $T_\epsilon$ is guaranteed by the model assumptions). The mapping $\varphi$ satisfies the assumptions of Theorem \ref{vineq} because the Kantorovich potential of $T_\epsilon$ is indeed an optimal Kantorovich potential from $T_0\# \mu_0$ to $\nu|\mu_0$ and each $T_\epsilon$ is the gradient of an $\alpha$-strongly convex function by assumption \ref{noise_map}. Furthermore, based on the same assumption, $\E (T_\epsilon) = \id$, and as a result, equation \eqref{potentials_integral} holds. 
	We can deduce that the mapping $\varphi_{\nu|\mu_0}$ is $\alpha(T_\epsilon)$-strongly convex. We can also observe that the value $\E\alpha(T_\epsilon)$ no longer depends on $\mu_0$. Additionally, we can deduce $\E \alpha(T_\epsilon)>0$ since each map $T_\epsilon$ is strongly convex.
	
	Collecting and combining the statements above, we can apply Theorem \ref{vineq} and show that for any $T\in \mathcal{T}$:
	
	\begin{equation}{\label{quadratic-growth-M}}
	\begin{split}
	M(T)-M(T_0) &=\frac{1}{2}\int\int [d^2_{\mathcal{W}}(T\#\mu_0,\nu)-d^2_{\mathcal{W}}(T_0\#\mu_0,\nu)] \diff P(\nu|\mu_0)\diff P_M(\mu_0)\\
	&\geq \frac{\E \alpha(T_\epsilon)}{2}\int  d^2_{\mathcal{W}}(T\#\mu_0,T_0\#\mu_0)\diff P_M(\mu_0).
	\end{split}
	\end{equation}
	Consequently, if for some $T$ we have that $M(T)=M(T_0)$,  inequality \eqref{quadratic-growth-M} implies that
	$$P_M\{\mu \text{ such that } d^2_{\mathcal{W}}(T\#\mu,T_0\#\mu)=0\} = 1.$$
	Since both $T_0$ and $T$ are optimal maps, whenever $d^2_{\mathcal{W}}(T\#\mu,T_0\#\mu)=0$ we can infer that $\norm{T-T_0}^2_{L^2(\mu)}=0$. Therefore
	$$P_M\{\mu \text{ such that } \norm{T-T_0}^2_{L^2(\mu)}=0\} = 1,$$
	which is equivalent to $\norm{T-T_0}^2_{L^2(Q)}=0$. Therefore, $T_0$ is the unique minimizer of $M$ up to $L^2(Q)$-norm, and hence identifiable up to $Q$-null sets.

\end{proof}

\begin{lemma}\label{embedding}
	If $\beta,\gamma>0$ are such that $\floor{\beta+\gamma} = \floor{\beta}$, then for any $f\in C^{\beta+\gamma}$ we have
	$$\norm{f}_{C^{\beta}} \leq 3 \norm{f}_{C^{\beta+\gamma}}.$$
\end{lemma}

\begin{proof}
	Recall that 
	\begin{equation}
	\norm{f}_{C^\beta}:= \max_{|\bk|\leq \floor{\beta}}\norm{D^{\bk}}_{\infty}+ \max_{|\bk|=\floor{\beta}} \sup_{x \neq y} \frac{ |D^{\bk}f(x)-D^{\bk}f(y)|}{\norm{x-y}^{\beta-\floor{\beta}}}.
	\end{equation}

	\noindent First, let's compare the expressions for $\norm{f}_{C^\beta}$ and $\norm{f}_{C^{\beta+\gamma}}$. For a function $g$, we have
	 $$\sup_{ x,y \in \Omega , x \neq y} \frac{ |g(x) -g(y)|}{\norm{x-y}^{\beta-b}}\le \sup_{ x,y \in \Omega , \norm{x -y}<1} \frac{ |g(x) -g(y)|}{\norm{x-y}^{\beta-b}}+\sup_{ x,y \in \Omega , \norm{x-y}\ge1} \frac{ |g(x) -g(y)|}{\norm{x-y}^{\beta-b}}.$$
	When $\norm{x-y}\geq 1$, we have $\norm{x-y}^{\beta-b}\geq 1$, therefore
	 $$\sup_{ x,y \in \Omega , \ge1} \frac{ |g(x) -g(y)|}{\norm{x-y}^{\beta-b}} \le \sup_{ x,y \in \Omega , \norm{x-y}\ge1}  |g(x) -g(y)| \le 2\|g\|_\infty,$$
	 $$\sup_{ x,y \in \Omega , \ge1} \frac{ |g(x) -g(y)|}{\norm{x-y}^{\beta-b}} \le \sup_{ x,y \in \Omega , \ge1}  |g(x) -g(y)| \le 2\|g\|_\infty,$$
	 $$\sup_{ x,y \in \Omega , \ge1} \frac{ |g(x) -g(y)|}{\norm{x-y}^{\beta-b}} \le \sup_{ x,y \in \Omega , \norm{x-y}\ge1}  |g(x) -g(y)| \le 2\|g\|_\infty,$$
	but whenever $\norm{x-y}< 1$, we have $\norm{x-y}^{\gamma}<1$, therefore $\norm{x-y}^{\beta+\gamma-b}\le\norm{x-y}^{\beta-b},$ so we obtain, $$\sup_{ x,y \in \Omega , \norm{x-y}<1} \frac{ |g(x) -g(y)|}{\norm{x-y}^{\beta-b}}\le \sup_{ x,y \in \Omega , \norm{x-y}<1} \frac{ |g(x) -g(y)|}{\norm{x-y}^{\beta+\gamma-b}}\le \sup_{ x,y \in \Omega , x \neq y} \frac{ |g(x) -g(y)|}{\norm{x-y}^{\beta+\gamma-b}}.$$
	It follows that
	\begin{equation}\label{holder-embedding}
		\sup_{ x,y \in \Omega , x \neq y} \frac{ |g(x) -g(y)|}{\norm{x-y}^{\beta-b}}\le 2\norm{g}_\infty+\sup_{ x,y \in \Omega , x \neq y} \frac{ |g(x) -g(y)|}{\norm{x-y}^{\beta+\gamma-b}}.
	\end{equation}
	Moreover, as $\floor{\beta+\gamma} = \floor{\beta}=b$, we have
	\begin{equation}\label{right-hand-side-holder}
	\max_{|\bk|\leq \floor{\beta+\gamma}}\norm{D^{\bk} f}_{\infty}=\max_{|\bk|\leq \floor{\beta}}\norm{D^{\bk} f}_{\infty}\geq \max_{|\bb|=b}\norm{D^{\bb} f}_\infty.
	\end{equation}
	
	Therefore, using inequality \eqref{right-hand-side-holder}, and inequality \eqref{holder-embedding} for a function $g$ of the form $D^{\bb} f$, where $\bb$ is a vector such that $|\bb|=b$, we obtain
	\begin{equation}
		\begin{split}
		\norm{f}_{C^{\beta}} &=  \max_{|\bk|\leq \floor{\beta}}\norm{D^{\bk} f}_{\infty}+\max_{|\bb|=b} \sup_{x \neq y} \frac{ |D^{\bb} f(x)-D^{\bb} f(y)|}{\norm{x-y}^{\beta-b}}\\
		&\leq \max_{|\bk|\leq \floor{\beta+\gamma}}\norm{D^{\bk} f}_{\infty}+\max_{|\bb|=b} \bigg[ 2\norm{D^{\bb} f}_\infty+
		\sup_{ x,y \in \Omega , x \neq y} \frac{ |D^{\bb} f(x) -D^{\bb} f(y)|}{\norm{x-y}^{\beta+\gamma-b}}\bigg]\\
		&\leq 3\norm{f}_{C^{\beta+\gamma}}.
		\end{split}
	\end{equation}
	
\end{proof}

\begin{proof}[Proof of Theorem \ref{existence_estimator}]
	\vspace{5mm}
	Suppose we have a sequence $\{T_n \in \mathcal{T}_{\beta,\gamma,R}\}$ that converges to a minimizer of $M_N$. Let us consider the corresponding sequence of convex potential functions,
	 $$\left\{\varphi_n: T_n = \nabla \varphi_n, \varphi_n \in \Phi \cap\overline{C^{\beta+\gamma}_R}^{\norm{.}_{C^\beta}}\right\}.$$ 
	 Since $C_R^{\beta+\gamma}$ is precompact in $C^{\beta}$ (as shown in Lemma 6.33 of \citet{gilbarg1977elliptic}), there exists a subsequence $\varphi_{n_k}$ converging to a function $\varphi$ in $C^{\beta}$. Moreover, since $\beta>2$, and convergence in $\beta$-H\"older norm implies convergence of second-order derivatives, and since $\{\varphi_n\} \subset \Phi$, we conclude $\varphi\in\Phi_0$.
	
	Since the norm $\norm{\cdot}_{C^{\beta}}$ is continuous with respect to its own induced topology, and $\norm{\varphi_n}_{C^\beta} \leq 3\norm{\varphi_n}_{C^{\beta+\gamma}}\leq 3 R$, we can also infer that $\norm{\varphi}_{C^{\beta}} \leq 3 R$.
		
	Next, let's consider the map $T=\nabla\varphi$ which consequently is in $\mathcal{T}_{\beta,3R}$. Given that the functional $M$ is continuous in $T$ and with respect to $L^2(Q)$-topology (which can be deduced using the triangle inequality and inequality \eqref{norm_comparison}), and since convergence in H\"older norm is stronger than convergence in $L^2(Q)$, we can conclude that the initial sequence of maps minimizing $M_N$, converges to $T$.	

\end{proof}

\begin{proof}[Proof of Lemma \ref{semi-metric}]
	It is trivial that $\rho(T,T)=0$ for any $T$. We can show that $\rho$ satisfies the triangle inequality as follows:
	\begin{equation*}
	\begin{split}
	& \big( \rho(T_1,T_2) + \rho(T_2,T_3)\big)^2\\
	&= \rho(T_1,T_2)^2+\rho(T_2,T_3)^2+2 \rho(T_1,T_2)\rho(T_2,T_3)\\
	&= \rho(T_1,T_2)^2+\rho(T_2,T_3)^2+2 \Big[\int d^2_{\mathcal{W}}(T_1\#\mu,T_2\#\mu)\diff P_M(\mu)\Big]^{1/2} \Big[\int d^2_{\mathcal{W}}(T_1\#\mu,T_2\#\mu)\diff P_M(\mu)\Big]^{1/2}\\
	&\geq \rho(T_1,T_2)^2+\rho(T_2,T_3)^2+2  \int d_{\mathcal{W}}(T_1\#\mu,T_2\#\mu)d_{\mathcal{W}}(T_1\#\mu,T_2\#\mu)\diff P_M(\mu) \\
	&= \int d^2_{\mathcal{W}}(T_1\#\mu,T_2\#\mu) + d^2_{\mathcal{W}}(T_2\#\mu,T_3\#\mu)+2 d_{\mathcal{W}}(T_1\#\mu,T_2\#\mu)d_{\mathcal{W}}(T_1\#\mu,T_2\#\mu)\diff P_M(\mu)\\
	&\geq \int d^2_{\mathcal{W}}(T_1\#\mu,T_3\#\mu)\diff P_M(\mu)\\
	&=\rho^2(T_1,T_3),
	\end{split}
	\end{equation*}
	where the first inequality is by Cauchy-Schwarz.
\end{proof}

To derive the convergence rate for the estimator, we will make use of some theorems from M-estimation \citep{van1996weak}. For the reader's convenience, we will first restate these theorems (with minor modifications to more easily relate to our context). Furthermore, we restate a  theorem from \citet{gunsilius2022convergence} that will be essential for determining the rate.

\begin{lemma}[Bracketing Entropy of H\"older Class, Corollary 2.7.4 \cite{van1996weak} ]{\label{metric-entropy-CL}}
Let $\Omega$ be a bounded, convex subset of $\R^d$ with a nonempty interior. There exists a constant $K$, depending only on $\beta$, $\text{vol}({\Omega})$, $r$ and $\rho$ such that,

$$\log N_{[]}(\epsilon,C^{\beta}_R(\Omega),L^r(Q))\leq KR^{d/{\beta}}{\epsilon}^{-d/{\beta}},$$
for every $r\geq 1$,$\epsilon>0$, and probability measure $Q$ on $\R^d$.
\end{lemma}

\begin{theorem}[\citet{van1996weak}, Theorem 3.2.5]{\label{theorem3.2.5}}
	Let $M_N$ be a stochastic process indexed by a semi-metric space $\Theta$ with semi-metric $\rho$, and let $M$ be a deterministic function, such that for every $\theta$ in a neighborhood of $\theta_0$,
	$$M(\theta)-M(\theta_0)\gtrsim d^2(\theta,\theta_0).$$
	Suppose that, for every $N$ and sufficiently small $\delta$,
	
	$$\E^* \sup_{d^2(\theta,\theta_0)<\delta} \sqrt{N}\big|(M_N-M)(\theta)-(M_N-M)(\theta_0)\big|\lesssim \phi_N(\delta),$$
	for functions $\phi_N$ such that $\delta \to \phi_N(\delta)/\delta^\alpha$ is decreasing for some $\alpha<2$ (not depending on $N$). Let
	$$r_N^2\phi_N\left(\frac{1}{r_N}\right)\leq \sqrt{N}, \quad \text{for every } N.$$
	If the sequence $\hat{\theta}_N$ satisfies $M_N(\hat{\theta}_N)\leq M_N(\theta_0)+O_P(r_N^{-2})$, and converges in outer probability to $\theta_0$, then $r_N d(\hat{\theta}_N,\theta_0)=O^*_P(1)$. If the displayed conditions are valid for every $\theta$ and $\delta$, then the condition that $\hat{\theta}_N$ is consistent is unnecessary.
	
\end{theorem}

\begin{theorem}[\citet{van1996weak}, Theorem 3.4.2]{\label{chaining}}
	Let $\mathcal{F}$ be a class of measurable functions such that $P f^2 < \delta^2$
	and $\norm{f}_{\infty}<M$ for every $f$ in $\mathcal{F}$. Then
	$$\E \sup_{f \in \mathcal{F}} \sqrt{N}| (\hat{P}-P)f|\leq \tilde{J}_{[]}(\delta,\mathcal{F},L^2(P))\Bigg(1+\frac{\tilde{J}_{[]}(\delta,\mathcal{F},L^2(P))}{\delta^2 \sqrt{N}} M \Bigg),$$
	where $\tilde{J}_{[]}(\delta,\mathcal{F},L^2(P))=\int_0^\delta \sqrt{1+\log N_{[]}(\epsilon,\mathcal{F},L^2(P))} \diff \epsilon$.
\end{theorem}

\begin{lemma}\label{bound-by-potentials}(Lemma 5 of \citet{gunsilius2022convergence})
	Let $\varphi_1, \varphi_2$ be proper strictly convex and bounded potential functions on every compact
	subset of $\Omega^0$ with Lipschitz-continuous gradients $\nabla\varphi_1$ and $\nabla\varphi_2$ satisfying  $\nabla\varphi_1(\Omega^0)= \nabla\varphi_2(\Omega^0)$. Then it holds for all $x \in \Omega^0$
	$$\norm{\nabla\varphi_1(x)-\nabla\varphi_2(x)}^2\leq c(1 + \max\{L_1,L_2\})^2 |\varphi_1(x)-\varphi_2(x)|$$
	where $0\leq L_1,L_2 < +\infty$ are the Lipschitz constants of $\nabla\varphi_1$  and $\nabla\varphi_2$ , respectively, $c < +\infty$ is a constant.
\end{lemma}

\begin{proof}[Proof of Theorem \ref{rate_of_convergence-d}]
	By observing that the right-hand side of inequality \eqref{quadratic-growth-M} is equal to $\frac{\alpha}{2}\rho(T,T_0)$, it follows that the functional $M$ demonstrates quadratic growth in the vicinity of its minimizer $T_0$ with respect to the semi-metric $\rho$. Therefore we can use the empirical process approach to obtain an upper bound for the rate of convergence.
	
	First, we find a function $\phi_N(\delta)$ such that
	\begin{equation*}
	\begin{split}
	\E \sup_{\rho(T,T_0)\leq \delta, T\in  \mathcal{T}_{\beta,\gamma,R} } \sqrt{N} \Big|(M_{N}-M)(T)-(M_{N}-M)(T_0)\Big|\leq \phi_N(\delta).
	\end{split}
	\end{equation*}
	Given that $\mathcal{T}_{\beta,\gamma,R}\subset \mathcal{T}_{\beta,3R}$ according to Lemma \ref{embedding}, we can instead find $\phi_N(\delta)$ such that
	\begin{equation}{\label{modulus of continuity}}
	\begin{split}
	\\E \sup_{\rho(T,T_0)\leq \delta, T\in  \mathcal{T}_{\beta,3R}}  \sqrt{N} \Big|(M_{N}-M)(T)-(M_{N}-M)(T_0)\Big|\leq \phi_N(\delta).
	\end{split}
	\end{equation}
	We define a class of functions indexed by $T$ as follows:
	$$\mathcal{F}_u:=\{f_T(\mu,\nu) = d^2_{\mathcal{W}}(T\#\mu,\nu)-d^2_{\mathcal{W}}(T_0\#\mu,\nu) , \text{  s.t.  } T\in  \mathcal{T}_{\beta,3R}  \text{ and } \rho(T,T_0)\leq u \},$$
	with the domain of each function $f_T \in \mathcal{F}_u$ being $\mathcal{W}_2(\Omega) \times \mathcal{W}_2(\Omega)$. We can see that \eqref{modulus of continuity} is equivalent to
	$$\E \sup_{f \in \mathcal{F}_\delta} \sqrt{N} |(P_N-P)f|\leq \phi_N(\delta).$$
    Denote by $\log N_{[]}(\epsilon,\mathcal{F}_u,L^2(P))$, the bracketing number of $\mathcal{F}_u$. We find an upper bound for this bracketing entropy using the bracketing entropy of the class of functions $C^\beta_{3R}$. To do this, note that any $f_T\in\mathcal{F}_u$ is induced by a map $T$ such that $T=\nabla \varphi$, for a convex function $\varphi \in C^{\beta}_{3R}$. Thus, for any $f_{T_1},f_{T_2} \in \mathcal{F}_u$, we have:
    \begin{equation}{\label{lipschitz}}
        \begin{split}
        	\norm{f_{T_1}-f_{T_2}}^2_{L^2(P)} &\leq \int |f_{T_1}(\mu,\nu)-f_{T_2}(\mu,\nu)|^2 \diff P(\mu,\nu)\\
        	&\leq\int|d^2_{\mathcal{W}}(T_1\#\mu,\nu)-d^2_{\mathcal{W}}(T_2\#\mu,\nu)|^2 \diff P(\mu,\nu)\\
            & \leq 4 \text{diam}(\Omega)^2 \int \norm{T_1-T_2}^2_{L^2(\mu)} \diff P_M(\mu)\\
            & \leq 4 \text{diam}(\Omega)^2  \norm{T_1-T_2}^2_{L^2(Q)}\\
            &\leq 4\text{diam}(\Omega)^2 C'\norm{\varphi_1-\varphi_2}_{L^1(Q)}.
        \end{split}
    \end{equation}
    To see why the last inequality holds, note that any $T\in \mathcal{T}_{\beta,3R}$ is the gradient of an $L$-smooth convex function, making it $L$-Lipschitz. Therefore, by applying Lemma \ref{bound-by-potentials} (our re-statement of \cite[Lemma 5]{gunsilius2022convergence}), we can establish that for $T_1,T_2\in \mathcal{T}_{\beta,3R}$ with corresponding potential functions $\varphi_1,\varphi_2$, the inequality $\norm{T_1-T_2}^2_{L^2(Q)}\leq c(1+L)^2 \norm{\varphi_1-\varphi_2}_{L^1(Q)}$ holds for some constant $c$.
     
    Using inequality \eqref{lipschitz} and Lemma \ref{metric-entropy-CL}, we obtain
    $$\log N_{[]}(\epsilon,\mathcal{F}_u,L^2(P))\lesssim\log N_{[]}(\epsilon/(4\text{diam}(\Omega)^2 C'),C^{\beta}_{3L}(\Omega),L^1(Q)) \lesssim \left(\frac{1}{\epsilon}\right)^{d/\beta}.$$
    By inequality (\ref{lipschitz}), we can also show that
    $$P f_T^2\leq P \norm{T-T_0}^2_{L^2(\mu)}=\norm{T-T_0}^2_{L^2(Q)}\leq u^2,$$
    for all $f_T \in \mathcal{F}_u$.

    Given that the functions in $\mathcal{F}_\delta$ are uniformly bounded over $(\mu,\nu)$ and $T$, and $Pf^2\leq \delta^2$ for all $f\in\mathcal{F}_\delta$, the conditions of Theorem \ref{chaining} are satisfied. This being the case, we can choose:
    $$\phi_N(\delta)=\tilde{J}_{[]}(\delta,\mathcal{F}_\delta,L^2(P))\Bigg(1+\frac{\tilde{J}_{[]}(\delta,\mathcal{F}_\delta,L^2(P))}{\delta^2 \sqrt{N}} \bar{c}\Bigg),$$
    where the constant $\bar{c}= 2 \text{diam}(\Omega)^2$ is a uniform upper bound for the functions in the class $\mathcal{F}_\delta$.  Using Lemma \ref{metric-entropy-CL} and when $\beta >\frac{d}{2}$, we can write
    \begin{equation}
        \begin{split}
            \tilde{J}_{[]}(\delta,\mathcal{F}_\delta,L^2(P))&:=\int_0^\delta \sqrt{1+\log N_{[]}(\epsilon,\mathcal{F}_\delta,L^2(P))} \diff \epsilon\\
            & \lesssim \int_0^\delta \sqrt{1+C \left(\frac{1}{\epsilon}\right)^{d/\beta} }\diff \epsilon\\
            & \lesssim \delta + \int_0^\delta \sqrt{\left(\frac{1}{\epsilon}\right)^{(d/\beta)}}\diff \epsilon \quad \quad \text{ since }\sqrt{1+a}\leq 1+\sqrt{a}  \quad \text{for } a\geq 0\\
            & \lesssim \delta + \frac{1}{1-\frac{d}{2\beta}}\delta^{(1-\frac{d}{2\beta})}\\
            & \lesssim  \delta ^{1-\frac{d}{2\beta}}.
        \end{split}
    \end{equation}
	When $\beta\leq\frac{d}{2}$, the bracketing integral above does not converge. However as stated in page 326 of \cite{van1996weak}, the above integral can be replaced by
	$$\int_{\min\{c\delta^2,\delta/3\}}^\delta \sqrt{1+\log N_{[]}(\epsilon,\mathcal{F}_\delta,L^2(P))} \diff \epsilon,$$
	where $c$ is a small postive constant, leading to
	\begin{center}
	\begin{equation}
		\int_{\min\{c\delta^2,\delta/3\}}^\delta \sqrt{1+\log N_{[]}(\epsilon,\mathcal{F}_\delta,L^2(P))} \diff \epsilon \lesssim
		\begin{cases} 
			\log(\frac{1}{\delta}) & \text{if } \beta = \frac{d}{2}, \\
			(\frac{1}{\delta})^{\frac{d-2\beta}{2\beta}} & \text{if } \beta < \frac{d}{2}.
		\end{cases}
	\end{equation}
\end{center}

For each of the three cases, we can use Theorem \ref{theorem3.2.5} to conclude the proof. The rate of convergence $r_N$ is obtained by the requirement $r_N^2 \phi_N\left(\frac{1}{r_N}\right)\lesssim \sqrt{N}$, which yields the following rates:
	\begin{equation*}
		r_N=
		\begin{cases} 
			N^{\frac{\beta}{2\beta+d}} & \text{if } \beta>\frac{d}{2}, \\[2ex]
			\frac{N^{1/4}}{(\log(N))^{1/2}} & \text{if } \beta = \frac{d}{2},\\[2ex]
			N^{\frac{1}{\beta(2\beta+d)}} & \text{if } \beta < \frac{d}{2}.
		\end{cases}
	\end{equation*}

\end{proof}

\begin{proof}[Proof of Theorem \ref{DC}]
	We first show that the functional $d^2_{\mathcal{W}}(T\# \mu,\nu)$ is expressible as the difference of two convex functionals for given measures $\mu$ and $\nu$
	\begin{equation*}
		\begin{split}
			d^2_{\mathcal{W}}(T\# \mu,\nu) &= \inf_{\pi \in \Gamma(T\#\mu,\nu)}\int \norm{z-y}^2 \diff \pi(z,y) \\
			&= \int \norm{z}^2 \diff T\#\mu(z) + \int \norm{y}^2 \diff \nu(y) - 2 \sup_{\pi \in \Gamma(T\#\mu,\nu)} \int \langle z,y \rangle \diff \pi(z,y) \\
			&= \int \norm{T(x)}^2 \diff \mu(x) + \int \norm{y}^2 \diff \nu(y) - 2 \sup_{\gamma \in \Gamma(\mu,\nu)} \int \langle T(x),y \rangle \diff \gamma(x,y)
		\end{split}
	\end{equation*}
	The convexity of the term $\int \norm{T(x)}^2 \diff \mu(x) + \int \norm{y}^2 \diff \nu(y)$ with respect to $T$ is evident. To demonstrate the convexity of the second term, consider any $S\in \mathcal{T}$ and $0<\alpha<1$.
	there exists a $\gamma^* \in \Gamma(\mu,\nu)$ satisfying
	\begin{equation*}
		\begin{split}
			\sup_{\gamma \in \Gamma(\mu,\nu)} \int \langle \alpha T(x) + &(1-\alpha) S(x),y \rangle \diff \gamma(x,y) = \int \langle \alpha T(x) + (1-\alpha) S(x),y \rangle \diff \gamma^*(x,y)\\
			&= \alpha \int \langle T(x),y \rangle \diff \gamma^*(x,y) + (1-\alpha) \int \langle S(x),y \rangle \diff \gamma^*(x,y)\\
			&\leq \alpha \sup_{\gamma \in \Gamma(\mu,\nu)} \int \langle T(x),y \rangle \diff \gamma(x,y) + (1-\alpha) \sup_{\gamma \in \Gamma(\mu,\nu)} \int \langle S(x),y \rangle \diff \gamma(x,y),
		\end{split}
	\end{equation*}
	confirming the convexity of the entire expression. Consequently, it follows that $d^2_{\mathcal{W}}(T\# \mu,\nu)$ represents the difference of two convex functionals, thereby extending the same conclusion to $M_N$.
\end{proof}

\begin{proof}[Proof of Lemma \ref{derivative-d}]
	
	Let $T\in \mathcal{T}$ and take any continuous function $\eta$ with domain $\Omega$. For $\epsilon>0$ sufficiently small, $T+\epsilon \eta$ is also in $\mathcal{T}$. If it exists, the Gateaux derivative of $d^2_{\mathcal{W}}(T\#\mu,\nu)$ is
	
	\begin{equation*}
	\begin{split}
	D_\eta d^2_{\mathcal{W}}(T\#\mu,\nu)&= \lim_{\epsilon \to 0} \frac{  d^2_{\mathcal{W}}((T+\epsilon \eta)\#\mu,\nu)-d^2_{\mathcal{W}}(T\#\mu,\nu) }{\epsilon}.
	\end{split}
	\end{equation*}
	In the following, we show that $ D_\eta d^2_{\mathcal{W}}(T\#\mu,\nu)$ exists and we calculate the limit. To do so, we will construct a coupling $\gamma$ between $\mu$ and $\nu$ with the property that $d^2_{\mathcal{W}}(T\#\mu,\nu)=\int  \norm{T(x)-y}^2 \diff \gamma(x,y)$. To this aim, let $S$ be an optimal map such that $S\#(T\#\mu)=\nu$. According to Brenier's theorem, such an optimal map exists and is unique if $T\#\mu$ is absolutely continuous. Since $T\in\mathcal{T}$ and $\mu$ is absolutely continuous, according to Lemma \ref{abs_pushforward}, so is $T\#\mu$. Let $X\sim \mu$ and define $Z=T(X)$ and $Y=S(Z)$. Denote by $\gamma$ the induced joint distribution of the pair $(X,Y)$. Note that  
	$$d^2_{\mathcal{W}}(T\#\mu,\nu)=\E \norm{Z-S(Z)}^2=\E \norm{Z-Y}^2=\E\norm{T(X)-Y}^2=\int \norm{T(x)-y}^2 \diff \gamma(x,y).$$
	Note that from the construction we can infer that $\gamma = (\id, S\circ T)\#\mu$ and thus the coupling $\gamma$ is independent of the random variable $X$.
	Thus:
	\begin{equation*}
	\begin{split}
	d^2_{\mathcal{W}}((T+\epsilon \eta)\#\mu,\nu) &\leq \int  \norm{(T+\epsilon\eta)(x)-y}^2\diff \gamma(x,y)\\
	&=   \int ( \norm{T(x)-y}^2 +2\epsilon\langle\eta(x),T(x)-y\rangle)\diff \gamma(x,y)+o(\epsilon^2)\\
	&=d^2_{\mathcal{W}}(T\#\mu,\nu)+2\epsilon\int \langle\eta(x),T(x)-y\rangle\diff \gamma(x,y)+o(\epsilon^2),
	\end{split}
	\end{equation*}
	and therefore
	$$\lim_{\epsilon \downarrow 0} \frac{  d^2_{\mathcal{W}}((T+\epsilon \eta)\#\mu,\nu)-d^2_{\mathcal{W}}(T\#\mu,\nu) }{\epsilon}\leq 2\int \langle\eta(x),T(x)-y\rangle\diff \gamma(x,y).$$
	
	Now define $\gamma_{\epsilon}$ using the same procedure as above, but such that $\gamma$ couples $\mu$ and $\nu$ while satisfying
	$d^2_{\mathcal{W}}((T+\epsilon\eta)\#\mu,\nu)=\int  \norm{T(x)+\epsilon\eta(x)-y}^2 \diff \gamma_{\epsilon}(x,y)$. Therefore similar to above we can see that $\gamma_{\epsilon}= (\id, (S_\epsilon \circ (T+\epsilon\eta))\# \mu$, where $S_{\epsilon}$ is the optimal map between $(T+\epsilon\eta)\#\mu$ and $\nu$. Thus
	
	\begin{equation*}
	\begin{split}
	d^2_{\mathcal{W}}((T+\epsilon \eta)\#\mu,\nu) &= \int \norm{(T+\epsilon\eta)(x)-y}^2\diff \gamma_{\epsilon}(x,y)\\
	&=   \int  (\norm{T(x)-y}^2 +2\epsilon\langle\eta(x),T(x)-y\rangle)\diff \gamma_{\epsilon}(x,y)+o(\epsilon^2)\\
	&\geq d^2_{\mathcal{W}}(T\#\mu,\nu)+2\epsilon\int \langle\eta(x),T(x)-y\rangle\diff \gamma_{\epsilon}(x,y)+o(\epsilon^2),
	\end{split}
	\end{equation*}
	and
	$$\lim_{\epsilon \downarrow 0} \frac{  d^2_{\mathcal{W}}((T+\epsilon \eta)\#\mu,\nu)-d^2_{\mathcal{W}}(T\#\mu,\nu) }{\epsilon}\geq \lim \inf_{\epsilon \downarrow 0} 2\int \langle\eta(x),T(x)-y\rangle\diff \gamma_{\epsilon}(x,y).$$
	
	To prove the existence of the limit, it is enough to show the integral with respect to $\gamma_\epsilon$ converges to the integral with respect to $\gamma$. We show the convergence of $\gamma_\epsilon$ to $\gamma$ in the Wasserstein metric.  Using the inequality (\ref{PC-d}) one can control the Wasserstein distance between the two measures by controlling $\norm{S_{\epsilon}\circ (T+\epsilon \eta)-S\circ T}^2_{L^2({\mu})}$. First note that as $\epsilon$ converges to zero, $(T+\epsilon \eta)\#\mu$ converges to $T\#\mu$, and again using the inequality (\ref{PC-d}) one can control the Wasserstein distance between the two measures by $\norm{\epsilon\eta}_{L^2(\mu)}^2$. As convergence in the Wasserstein distance results in narrow convergence, one can show that $S_{\epsilon}$ converges to $S$  using Theorem 1.7.7 in Panaretos \& Zemel \cite{panaretos2020invitation}. By Lemma \ref{abs_pushforward}, $(T+\epsilon \eta)\#\mu$ and $T\#\mu$ are absolutely continuous, thus using Theorem 5.20 in Villani \cite{villani2009optimal}  $S_\epsilon$ and $S$ are continuous. Therefore $S_{\epsilon}\circ (T+\epsilon \eta)\to S\circ T$ and we can conclude $\gamma_{\epsilon}\to \gamma$. 
	Additionally, $T$ is bounded and continuous, therefore we can conclude the integral with respect to $\gamma_\epsilon$ converges to the integral with respect to $\gamma$.
	The two inequalities prove the existence of the derivative and:
	$$D_\eta d^2_{\mathcal{W}}(T\#\mu,\nu)= \lim_{\epsilon \to 0} \frac{  d^2_{\mathcal{W}}((T+\epsilon \eta)\#\mu,\nu)-d^2_{\mathcal{W}}(T\#\mu,\nu) }{\epsilon}=2\int \langle\eta(x),T(x)-y\rangle\diff \gamma(x,y).$$
	Thus
	\begin{equation}
	\begin{split}
	& D_\eta M_N(T)=\frac{1}{N}\sum_{i=1}^N \int \langle\eta(x),T(x)-y\rangle\diff \gamma_{\mu_i,\nu_i}(x,y)\\
	&\text{for }\gamma_{\mu_i,\nu_i}\in \Gamma(\mu_i,\nu_i)  \quad \text{s.t.} \quad    d^2_{\mathcal{W}}(T\#\mu_i,\nu_i)=\int  \norm{T(x)-y}^2 \diff \gamma_{\mu_i,\nu_i}(x,y),
	\end{split}
	\end{equation}
	and
	
	\begin{equation}
	\begin{split}
	& D_\eta M(T)=\int \int \langle\eta(x),T(x)-y\rangle\diff \gamma_{\mu,\nu}(x,y)\diff P(\mu,\nu)\\
	&\text{for }\gamma_{\mu,\nu}\in \Gamma(\mu,\nu)  \quad \text{s.t.} \quad  d^2_{\mathcal{W}}(T\#\mu_i,\nu_i)=\int  \norm{T(x)-y}^2 \diff \gamma_{\mu,\nu}(x,y).
	\end{split}
	\end{equation}
\end{proof}

\begin{proof}[Proof of Lemma \ref{derivative-gaussian}]
	We are essentially seeking to find equivalent expressions for \eqref{D_M_N} and \eqref{D_h} in the Gaussian case. 
	It suffices to note that for linear optimal maps $T$ and $S$, and a Gaussian measure $\mu=N(0,\Sigma)$, we have
	$$\langle T,S \rangle_{L^2(\mu)} = \int_{\R^d} \langle Tx,Sx \rangle \diff \mu(x) = tr(A\Sigma B).$$
	For the proof see \citet{takatsu2011wasserstein}.
	Using this, we can directly find the expression in the Lemma.
\end{proof}

\begin{proof}[Proof of Theorem \ref{rate_of_convergence-gaussian}]
	Similar to the proof of Theorem \ref{rate_of_convergence-d} we can establish the quadratic growth of $M$ in the vicinity of its minimizer $T_0$ with respect to the semi-metric $\rho$. Therefore we can again use the empirical process approach to obtain an upper bound for the rate of convergence.
	First, we find a function $\phi_N(\delta)$ such that
	\begin{equation*}
	\begin{split}
	\E \sup_{\rho(T,T_0)\leq \delta, T \in S^d_{++}} \sqrt{N} \Big|(M_{N}-M)(T)-(M_{N}-M)(T_0)\Big|\leq \phi_N(\delta).
	\end{split}
	\end{equation*}
	Denote by $\alpha_0$ the minimum eigen value of $T_0$. By assumption $\alpha_0>0$. Now based on Theorem 6 of \citet{manole2021plugin} for any absolutely continuous probability measure $\mu$ on $\R^d$ and any $T\in S^d_{++}$,
	$$\frac{1}{\alpha_0^2}\norm{T-T_0}^2_{L^2(\mu)}\leq d^2_{\mathcal{W}}(T\#\mu,T\#\mu),$$
	and hence 
	$$\frac{1}{\alpha_0^2}\norm{T-T_0}^2_{L^2(Q)}\leq \rho^2(T,T_0).$$
	Therefore, we can instead find a function $\phi_N(\delta)$ such that
	\begin{equation*}
		\begin{split}
		\E \sup_{\norm{T-T_0}_{L^2(Q)} \leq \alpha_0 \delta, T \in S^d_{++} } \sqrt{N} \Big|(M_{N}-M)(T)-(M_{N}-M)(T_0)\Big|\leq \phi_N(\delta).
		\end{split}
		\end{equation*}
	Let $\eta = T-T_0$. By mean value theorem for Gateaux differentiable functions (see Theorem 4.1.2 of \citep{kurdila2006convex}), we can write:
	\begin{equation}
		\begin{split}
			|M_{N}-M)(T) - (M_{N}-M)(T_0)|\leq \sup_{0\leq t \leq 1} |D_\eta (M_{N}-M)(T_0+t \eta)| \norm{\eta}_{L^2(Q)},
			\end{split}
	\end{equation}
	For any matrix $S\in S^d_{++} $, $D_\eta M_{N}(S)$ is a sum of independent random variables $2\text{tr}(\eta M_i (S-G_{i} S))$, where $G_i$ are the optimal maps between $SM_iS$ and $N_i$. Sinc the eigen values of all the matrices in this expression are bounded from below and above, the trace is also bounded. Therefore central limit theorem applies and we can write
	$$\E_P  |D_\eta (M_{N}-M)(S)| \lesssim \frac{1}{\sqrt{N}}.$$
	By plugging the inequality into the expression above we obtain    
	\begin{equation*}
		\begin{split}
			\E \sup_{\norm{T-T_0}_{L^2(Q)} \leq \alpha_0 \delta, T \in S^d_{++} } \sqrt{N}\Big|(M_{N}-M)(T)-(M_{N}-M)(T_0)\Big| &\leq \norm{T-T_0}_{L^2(Q)}\leq \alpha_0 \delta .
		\end{split}
	\end{equation*}
	and we conclude $\phi_N(\delta)=\delta$ and the rate of convergence is $N^{-1/2}$.
\end{proof}

\section{Generating random convex functions}{\label{Generation of convex functions}}

The purpose of this section is to briefly discuss how one might numerically generate random functions $\varphi(x, y)$ that are convex on a domain $U = [x_0, x_1]^2$ and average to $(x^2 + y^2)/2$ (therefore their gradient is an optimal map and their average is the $\id$ map, as required in our Model). This can be easily done with linear maps, but we are interested in more variability and complex maps. To this aim, we take functions of the form
\begin{equation}
\varphi = \frac{x^2+y^2}{2} + \sum_{d=2}^D \sum_{k=0}^d \frac{1}{k!(d-k)!} \left[ a_{d, k} x^k y^{d-k} +  b_{d, k} x^{1/k} y^{1/(d-k)} \right],
\end{equation}
where $a_{d, k}$ and $b_{d, k}$ are random coefficients. To avoid the singularity of the second term, we take $(x_0, x_1) = (0.5, 1.5)$. We require that
$\E[\nabla \varphi] = (x, y)^{\top}$
which implies $ \E[a_{d, k}] = \E[b_{d, k}] = 0$. Our considerations in this section being practical,  we shall probe for convexity numerically and approximately (at least in probability). Namely, we set $D = 8$ and we take the coefficients $a_{d, k}$ and $b_{d, k}$ to be independent random variables with a centred normal distribution of variance $\sigma^2$. Using the \texttt{sympy} library of Python, we explicitly calculate the expectation of the determinant of the Hessian matrix, and we calculate the minimum in the domain at $(1.5, 1.5)$, specifically
\begin{equation}
\min_{x,y \in U} \E\left[\det(\nabla^2 \varphi)\right] \approx 1-11\sigma^2.
\end{equation}
Similarly, we estimate the maximum variance of the Hessian determinant to be at most
\begin{equation}
\Var\left[\det(\nabla^2 \varphi)\right] \lesssim 10\sigma^2,
\end{equation}
for small enough $\sigma$. Assuming the determinant follows a normal distribution, the probability of generating a non-convex function $\mathbb{P}\left\{\det(\nabla^2 \varphi) < 0\right\}$ decreases exponentially below $\sigma \lesssim 0.2$, being around $0.1$ for $\sigma = 0.15$ and $\sim 10^{-25}$ for $\sigma = 0.03$, which is the parameter used to generate the noise maps in Fig.~\ref{fig:flow}. A way to limit the probability of generating a non-convex function even further could be to use a distribution with support on a finite domain like a beta distribution adjusted to the domain $U$, instead of a normal distribution.

\bibliographystyle{imsart-nameyear}
\bibliography{complex}

\end{document}